\documentclass{article}
\usepackage[utf8]{inputenc}
\usepackage[T1]{fontenc}
\usepackage[latin9]{luainputenc}
\usepackage{geometry}
\usepackage{babel}
\usepackage{amsmath}
\usepackage{amsthm}
\usepackage{amssymb}
\usepackage[scr=esstix]
       {mathalpha}
\usepackage[unicode=true,pdfusetitle,
 bookmarks=true,bookmarksnumbered=false,bookmarksopen=false,
 breaklinks=false,pdfborder={0 0 1},backref=false,colorlinks=false]
 {hyperref}
\usepackage{dsfont,graphicx}
\usepackage[dvipsnames]{xcolor}

\usepackage{esint}

\usepackage{physics}
\usepackage{amsmath}
\usepackage{tikz}
\usepackage{mathdots}
\usepackage{yhmath}
\usepackage{cancel}
\usepackage{color}
\usepackage{siunitx}
\usepackage{array}
\usepackage{multirow}
\usepackage{gensymb}
\usepackage{tabularx}
\usepackage{extarrows}
\usepackage{booktabs}
\usetikzlibrary{fadings}
\usetikzlibrary{patterns}
\usetikzlibrary{shadows.blur}
\usetikzlibrary{shapes}

\makeatletter
\numberwithin{equation}{section}
\numberwithin{figure}{section}
\theoremstyle{plain}
\newtheorem{thm}{\protect\theoremname}[section]
\theoremstyle{plain}
\newtheorem{prop}[thm]{\protect\propositionname}
\theoremstyle{plain}
\newtheorem{cor}[thm]{\protect\corollaryname}
\theoremstyle{plain}
\newtheorem{lem}[thm]{\protect\lemmaname}
\theoremstyle{definition}
\newtheorem{defn}[thm]{\protect\definitionname}
\theoremstyle{remark}
\newtheorem{rem}[thm]{\protect\remarkname}

\@ifundefined{date}{}{\date{}}

\def\vep{\varepsilon}

\def\be{\begin{equation}}
\def\ee{\end{equation}}

\def\ot{{\overline{t}}}
\def\R{{\mathbb{R}}}
\def\E{{\mathbb{E}}}
\def\P{{\mathbb{P}}}
\def\ut{{\underline{t}}}
\usepackage[nottoc,notlot,notlof]{tocbibind}

\makeatother
\title{Free boundary regularity and well-posedness of physical solutions to the supercooled Stefan problem}
\providecommand{\corollaryname}{Corollary}
\providecommand{\definitionname}{Definition}
\providecommand{\lemmaname}{Lemma}
\providecommand{\propositionname}{Proposition}
\providecommand{\remarkname}{Remark}
\providecommand{\theoremname}{Theorem}
\author{Sebastian Munoz \thanks{Department of Mathematics, University of California, Los Angeles, 90095, USA. \texttt{sebastian@math.ucla.edu}}}
\begin{document}
\maketitle
\begin{abstract}
    We study the regularity and well-posedness of physical solutions to the supercooled Stefan problem. Assuming only that the initial temperature is integrable, we prove that the free boundary, known to have jump discontinuities as a function of the time variable, is $C^1$ as a function of the space variable, and is $C^{\infty}$ in an open set with locally finite complement, which we describe explicitly. We also prove that, as conjectured in \cite{delarue}, the set of positive times when a jump occurs cannot have accumulation points. In addition, we prove that short-time uniqueness of physical solutions implies global uniqueness, which allows us to obtain uniqueness for very general initial data that fall outside the scope of the current well-posedness regime. In particular, we answer two questions left open in \cite{LedSoj,shkolnikov} regarding the global uniqueness of solutions. We proceed by deriving a weighted obstacle problem satisfied by the solutions, which we exploit to establish regularity and nondegeneracy estimates and to classify the free boundary points. We also establish a backward propagation of oscillation property, which allows us to control the occurrence of future jumps in terms of the past oscillation of the solution.
\end{abstract}

\tableofcontents
\section{Introduction}
We study the regularity and well-posedness of the supercooled Stefan problem
\begin{equation}\begin{cases} \label{supercooled eqs}
    u_t - \frac{1}{2} u_{xx}= 0 & t>0,\; x>\Lambda_t,\\
\dot{\Lambda}_t = \frac{\alpha}{2} u_x(\Lambda_t,t) & t > 0,\\
u(x,0) = u_0(x),\; u(\Lambda_t,t) = 0 & x\geq 0,\; t\geq 0,
\end{cases}
\end{equation}
where $\alpha>0$, $u_0\geq0$ is a probability density, and $\Lambda :[0,\infty)\to [0,\infty)$ is a nondecreasing function.

The problem \eqref{supercooled eqs} models the freezing of supercooled water into ice, with initial temperature $-u_0$, in the one-phase setting where the temperature in the solid is held constant, and equal to $0$. Unlike the classical Stefan problem, \eqref{supercooled eqs} is known to be ill-posed and its classical solutions may in general blow up in finite time, leading to discontinuities in the frontier $\Lambda$ and instantaneous temperature spikes (see \cite{sherman}). Therefore, the problem of defining the solution beyond the blow-up is of central interest.

For this purpose, \eqref{supercooled eqs} can be reformulated in a probabilistic sense. From this viewpoint, one considers a large population of Brownian particles that interact through the moving boundary: as particles hit the freezing threshold, they stop (modeling solidification) and release heat that instantaneously affects the remaining particles. This leads to the McKean-Vlasov equation
\begin{equation}
    \begin{cases}X_t=X_{0-}+B_t-\Lambda_t, \\
\tau=\inf\{t\geq 0: X_t \leq 0\},\\
\Lambda_t=\alpha \P(\tau \leq t), 
 \end{cases}
\end{equation}
where $B_t$ is a standard Brownian motion, and $\mathcal{L}(X_{0-})=u_0(x)dx$. Beyond the physical setting of the Stefan problem, this type of model has also been studied in the context of contagion in large financial markets (see \cite{hambly,HamSoj,LedSoj2,NadShk,NadShk2}) and integrate-and-fire models in neuroscience (see \cite{CarPerSalSme,DelIngRubTan,IngTal}). An advantage of this formulation is that it allows for the moving front $\Lambda$ to have jump discontinuities. We will work with the following equivalent definition, expressed directly in terms of Brownian motion, which is more amenable to our methodology.

\begin{defn} \label{def: probabilistic} Given a probability density $u_0$ on $[0,\infty)$, a standard Brownian motion $B$ on $\R$ with initial distribution $u_0(x)dx$, and a nondecreasing, right-continuous function $\Lambda:[0,\infty) \to [0,\infty)$, we say that $\Lambda$ is a probabilistic solution to the supercooled Stefan problem \eqref{supercooled eqs} if
\begin{equation} \label{eq: probabilistic lambda}
\Lambda_t=\alpha\P(\tau \leq t), \quad  t\geq0, \;\;\tau=\inf\{t\geq0: B_t \leq \Lambda_t\}.
\end{equation}\end{defn}

The way to recover the original system \eqref{supercooled eqs} from a probabilistic solution, at least formally, is by letting $u(x,t)$ be the density function of the Brownian motion before it stops. Namely, if $\Lambda$ is a probabilistic solution, and $u:\R \times (0,\infty) \to \R$ is given by
\begin{equation} \label{u defi}    \P(B_t\in [a,b], t<\tau)=\int_{a}^bu(x,t)dx,\quad \quad a<b,\; \;t>0,
\end{equation}
then $u$ satisfies \eqref{supercooled eqs}, as long as the solution is sufficiently regular. However, in the presence of jump discontinuities of $\Lambda$, the classical formulation breaks down, and the choice of the jump size is not uniquely determined. This motivates the following definition.
\begin{defn} \label{def: physical} Let $\Lambda$ be a probabilistic solution to \eqref{supercooled eqs}, with the convention $\Lambda_{0-}:=0$. We say that $\Lambda$ is a physical solution if, for every $t\geq 0$,
\begin{equation} \label{eq: physical defi}
\Lambda_{t}-\Lambda_{t-}=\inf\{x>0: \P(t \leq \tau, B_{t}\in (\Lambda_{t-},\Lambda_{t-}+x])<\alpha^{-1}x\}.
\end{equation}
\end{defn}
The notion of physical solution is a natural one, as it can be shown that these are the solutions for which the jump distance is as small as possible (see \cite[Prop. 1.2]{hambly}). Existence of physical solutions was recently shown by C. Cuchiero, S. Rigger, and S. Svaluto-Ferro in \cite[Thm. 6.5]{CRF} under very general assumptions, requiring only a finite moment for the initial density $u_0$ (see also \cite{DelIngRubTan,NadShk,NadShk2}).  Even more recently, this finite-moment assumption was removed by S. Ledger and A. S\o jmark, who constructed physical solutions for finite initial measures \cite[Thms.~4.1, 4.7 and Rem.~4.8]{LedSojTransport}. Uniqueness was shown first by F. Delarue, S. Nadtochiy, and M. Shkolnikov in \cite[Thm. 1.1]{delarue}, under the assumption that $u_0$ is a bounded function that changes monotonicity finitely often on compact sets. More recently, uniqueness was obtained for a broad class of oscillatory initial conditions in \cite[Thm. 1.1]{shkolnikov}, albeit requiring the hard $L^{\infty}$ constraint $u_0\leq \alpha^{-1}$. Local uniqueness was shown for very general initial data by Ledger and S\o jmark in \cite[Thm. 3.1]{LedSoj}, requiring only a natural assumption for $u_0$ near $x=0$.

In this paper, we study the regularity of physical solutions and their free boundary, while making no assumptions on the initial density $u_0$. Therefore, our regularity results apply equally to the locally monotone regime of \cite{delarue}, to the oscillatory regime of \cite{shkolnikov}, and to the possibly non-unique solutions of \cite{CRF}. We also advance the well-posedness theory considerably, by proving global uniqueness of physical solutions for initial data that fall outside the scope of \cite{delarue,shkolnikov}, under only the local condition of \cite[Thm. 3.1]{LedSoj}.

Our first main result answers an open question from \cite{delarue}, namely whether the time discontinuities of $\Lambda$ can have accumulation points. It states that no positive time can be an accumulation point for the jump times. 
\begin{thm} \label{thm: jumps}Assume that $u_0$ is a probability density on $[0,\infty)$, and let $\Lambda$ be a physical solution to \eqref{supercooled eqs}. Then the set of discontinuities of $\Lambda$ is locally finite in $(0,\infty)$.    
\end{thm}
For the locally monotone solutions of \cite[Thm. 1.1]{delarue}, it is already known that $t=0$ cannot be an accumulation point, so Theorem \ref{thm: jumps} implies that the jumps do not accumulate at all. On the other hand, if the initial data is highly oscillatory, as in \cite{shkolnikov}, our result says that these oscillations are resolved instantly, so that the jumps can only accumulate at the initial time. 

Next, we present our results concerning the regularity of the free boundary. Since the function $\Lambda_t$ may have infinite speed, and even jump discontinuities, any global regularity statement requires representing the boundary as a function of the space variable. To state the theorem, we define the freezing time function $s:\R \to [0,\infty]$ by
\begin{equation} \label{s defi}
    s(x)=\inf\{t>0: \Lambda_t>x\}.
\end{equation}
The freezing time satisfies $s(x)\equiv 0$ for $x\leq \Lambda_0$, and $s(x)\equiv \infty$ for $x\geq \alpha$. For $x\in (\Lambda_0,\alpha)$, one has $s(x)\in (0,\infty)$, and $s$ is a left inverse for $\Lambda$, namely
\[s(\Lambda_t)=t,\quad t\in (0,\infty).\]
Indeed, since $\Lambda$ can be shown to be strictly increasing, $s$ is continuous (see Lemma \ref{lem: lambda strict incr}). The jump discontinuities of $\Lambda$ then correspond to intervals in which $s$ is constant. Our first regularity result says that the free boundary, viewed as the graph of the freezing time, is of class $C^1$, and it is smooth outside of a locally finite set.
\begin{thm}\label{thm: C1 regu} Assume that $u_0$ is a probability density on $[0,\infty)$, let $\Lambda$ be a physical solution to \eqref{supercooled eqs}, and let $s$ be given by \eqref{s defi}. Then $s\in C^{1}((\Lambda_0,\alpha))$, and there exists an open set $R \subset (\Lambda_0,\alpha)$ such that $(\Lambda_0,\alpha)\backslash R$ is locally finite, with $s \in C^{\infty}(R)$.
\end{thm}
To give a more precise description of the free boundary, we classify the elements of $(\Lambda_0,\alpha)$ into regular points and singular points, as follows. Let $J \subset (\Lambda_0,\alpha)$ be the set of points located strictly inside a jump, namely
\begin{equation}
J:=\{x\in \R  : \Lambda_{t-}<x<\Lambda_t \text{ for some } t\in (0,\infty) \}=\{x\in (\Lambda_0,\alpha)  : \Lambda_{s(x)-}<x<\Lambda_{s(x)} \}.
\end{equation}

In the following definition and throughout the paper, whenever pointwise values of the density $u$ defined by \eqref{u defi} are used, we take $u$ to be the canonical representative specified in Lemma \ref{lem: u bound}.

\begin{defn} \label{def: reg sing} Let $x_0 \in (\Lambda_0,\alpha)$. We say that $x_0$ is a regular point if \begin{equation} u(x_0,s(x_0)-):=\lim_{t\uparrow s(x_0)}u(x_0,t)=0 \quad\quad  \text{or} \quad\quad  x_0\in J ,\end{equation} and we say that $x_0$ is a singular point if 
\begin{equation}x_0\notin J,\;\; u(x_0,s(x_0)-)=\alpha^{-1} \quad \quad\quad \text{or}\quad\quad \quad \Lambda_{s(x_0)-}<\Lambda_{s(x_0)}, \;\; x_0 \in \{\Lambda_{s(x_0)-},\Lambda_{s(x_0)}\}.\end{equation} Let $R$ denote the set of regular points, and $S$ denote the set of singular points. If $x_0 \in R$, we say that $(x_0,s(x_0))$ is a regular free boundary point, and if $x_0\in S$, we say that $(x_0,s(x_0))$ is a singular free boundary point. Finally, if $t\in s(S)$, we say that $t$ is a singular time.
\end{defn}
The regular free boundary points are therefore those where either $u$ vanishes continuously or those located strictly inside a jump. The singular free boundary points are either the endpoints of a jump, or points where $\Lambda$ is continuous but $u$ jumps down to $0$ from the critical value $\alpha^{-1}$. This definition is justified by the next result, which guarantees that the sets $R$ and $S$ are disjoint, and they exhaust all possible free boundary points, so that $R \cup S=(\Lambda_0,\alpha)$. 
\begin{thm}\label{thm: classification}Assume that $u_0$ is a probability density on $[0,\infty)$, let $\Lambda$ be a physical solution to \eqref{supercooled eqs}, let $u$ be given by \eqref{u defi}, and let $s:(\Lambda_0,\alpha)\to (0,\infty)$ be given by \eqref{s defi}. Let $R$ be the set of all regular points, and let $S$ be the set of all singular points. Then $R\cap S=\emptyset$ and $R\cup S=(\Lambda_0,\alpha)$. Moreover, the set of singular times $s(S)$ is locally finite in $(0,\infty)$, $S$ is locally finite, the set $R$ is open, $s\in C^{\infty}(R)$, and $u$ extends smoothly up to the regular part of the free boundary,
\begin{equation} \label{u smooth} \quad u\in C^{\infty}(\{(x,t): x\in R,\; t\leq s(x)\}).\end{equation}  Finally, if one lets
\begin{equation}U:=R\backslash J=\{x\in (\Lambda_0,\alpha): u(x,s(x)-)=0\},    \end{equation} then $U$ is open, \[U=\{x\in(\Lambda_0,\alpha): s'(x)>0\}, \quad \Lambda\in C^{\infty}(s(U)), \] and
\begin{equation} \label{classical speed formula}
    \dot \Lambda_{s(x)}=\frac{1}{s'(x)}=\frac{\alpha}{2}u_x(\Lambda_{s(x)},s(x)),\quad x\in U.
\end{equation}
\end{thm}
This result provides a thorough classification of the free boundary points, guaranteeing that the solution and its free boundary are smooth outside of an explicitly described locally finite set. In particular, it implies a strictly stronger statement than Theorem \ref{thm: jumps}: the set of all singular times, including those at which $\Lambda$ is continuous, has no accumulation points. The latter singularities can be understood as ``degenerate jumps'' of zero length, and our proof methodology treats them as such (see Section \ref{sec: jumps}).

To place Theorem \ref{thm: classification} in context, we compare it with the regularity result of \cite[Thm. 1.1]{delarue}, which was proved under the assumption that $u_0$ is bounded and changes monotonicity finitely often on compact intervals. In that result, the authors showed that for each time $t_0$, there are three possibilities: (i) if \(u(\Lambda_{t_0-},t_0-)=0\) and $u_x(\Lambda_{t_0-},t_0-)<\infty$, then \(\Lambda\) is \(C^{1}\) on some interval \([t_{0},t_{0}+\varepsilon)\);
(ii) if $u_x(\Lambda_{t_0-},t_0-)=\infty$ but \(u(\Lambda_{t_0-},t_0-)<\alpha^{-1}\) then \(\Lambda\) is \(C^{1/2}\) on \([t_{0},t_{0}+\varepsilon)\);
(iii) if \(u(\Lambda_{t_0-},t_0-)\ge \alpha^{-1}\), then \(\Lambda\) is neither \(C^{1}\) nor \(C^{1/2}\) at \(t_{0}\), and a jump may occur at $t_0$.

Theorem \ref{thm: classification} shows that the intermediate case (ii) never occurs at any point of continuity $t_0$ of $\Lambda$. In fact, one always has $u(\Lambda_{t_0},t_0-)\in \{0,\alpha^{-1}\}$ at such a point, so the frontier $\Lambda$ is automatically smooth near $t_0$ whenever $\lim_{t \uparrow t_0}u(\Lambda_{t_0},t)<\alpha^{-1}.$ Jump times are different: once the jump has occurred, the post-jump profile may fall under the $C^{1/2}$ alternative of \cite[Props.~2.1--2.2]{delarue}. Moreover, when $t_0\in s(S)$, our result gives maximal regularity on the punctured neighborhood $(t_0-\vep,t_0)\cup(t_0,t_0+\vep)$, rather than only a one-sided interval $(t_0,t_0+\vep)$. This upgrade is possible in part thanks to Theorem \ref{thm: stronger jumps}, which rules out the accumulation of singular times to the left of $t_0$. For a more precise discussion of the one-sided behavior at jump endpoints, we refer to Remark \ref{rem:vanishing-rates}. We also highlight that Theorem \ref{thm: classification} is valid without the requirement that $u_0$ is bounded or locally monotone.

Our last main result is a uniqueness claim.
\begin{thm} \label{thm: local to global uniqueness}
    Assume $u_0$ is a probability density, let $t_0>0$, and assume that any two physical solutions to \eqref{supercooled eqs} agree on $[0,t_0]$. Then there exists at most one physical solution to \eqref{supercooled eqs}.
\end{thm}
We obtain this result by showing that physical solutions automatically become locally monotone at any positive time $t_0>0$ (see Proposition \ref{prop: monotonicity changes}), even if $u_0$ fails to satisfy this assumption. Therefore, uniqueness on $[t_0,\infty)$ follows from the existing well-posedness theory of \cite[Thm. 1.1]{delarue}. We now present two consequences of Theorem  \ref{thm: local to global uniqueness}. The first of these is a uniqueness result under the assumption that, near $x=0$, $u_0$ does not oscillate or become too flat if it approaches the critical value $\alpha^{-1}$.  Local uniqueness under this assumption was proved in \cite[Thm. 3.1]{LedSoj}, but the global uniqueness was left as an open question, which is answered by the following corollary.
\begin{cor} \label{cor: unique xn} Assume that $u_0$ is a probability density, and assume that there exist constants $\delta>0$, $c>0$, and $n\in \mathbb{N}$ such that 
\[u_0(x)\leq \alpha^{-1} -c x^n \quad \text{for almost every}  \; x\in (0,\delta).\]
Then \eqref{supercooled eqs} has at most one physical solution.
\end{cor}
The main novelty of Corollary \ref{cor: unique xn} is that it provides {\it global} uniqueness under only {\it local} assumptions on $u_0$ near $x=0$.

To understand this result, it is also important to note that oscillation around the value $\alpha^{-1}$ near $x=0$ is the critical behavior when it comes to uniqueness. Indeed, if $\text{ess}\limsup_{x\downarrow 0}u_0(x)<\alpha^{-1}$, then Corollary \ref{cor: unique xn} automatically gives global uniqueness. On the other hand, if $\text{ess}\liminf_{x\downarrow 0}u_0(x)>\alpha^{-1}$, then the solution jumps at $t=0$, and the physical jump condition \eqref{eq: physical defi} then forces $\text{ess}\liminf_{x\downarrow\Lambda_{0}}u_0(x)\leq\alpha^{-1}$.

As a second application, Theorem \ref{thm: local to global uniqueness} yields uniqueness for a family of oscillatory initial densities that arise as a critical case of the oscillatory regime of \cite[Thm. 1.1]{shkolnikov}. Indeed, it was shown in \cite[Prop. 1.6]{shkolnikov} that the solutions are locally unique, but the global uniqueness was again posed as an open question, addressed by our next result.
\begin{cor}Assume that $\alpha=1$, and assume that
\begin{equation}
u_0(x)=
\begin{cases}
\alpha_1,& x\in\bigcup_{n\ge 1}[a_{2n},a_{2n-1}),\\
\alpha_2,& x\in\bigcup_{n\ge 1}[a_{2n+1},a_{2n}),
\end{cases}
\end{equation}
where $0<\alpha_1<1<\alpha_2$, $a_{2n-1}=r^{n-1}a_1$, $a_{2n}=pr^{n-1}a_1$, and $r=pq$, $p,q\in(0,1)$. Then, for any $\alpha_2>1$ sufficiently close to 1, the physical solution $\Lambda$ of \eqref{supercooled eqs} is unique.  
\end{cor}
We now describe our methodology, which begins, in Sections \ref{sec:estimates} and \ref{sec: estimates physical}, by analyzing a partial differential equation solved by any probabilistic solution, namely
\begin{equation}\label{u eq intro} u_t-\frac12 u_{xx}=(\nu(x)\chi_{u>0})_t,\end{equation}
where $\nu$ is the probability density of the stopped particles, so that $B_{\tau} \sim \nu(x)dx.$ If we define $u(\cdot,0-):=u_0$, then $\nu$ can be described explicitly as
\begin{equation} \label{nu defi}
    \nu(x)=\begin{cases} 0 &\text{if} \;x\notin(0,\alpha),\\ 
    u(x,t-) & \text{if} \;\;\Lambda_{t-}<x<\Lambda_t \; \text{for some }\; t\in [0,\infty),\\
    \alpha^{-1} & \text{otherwise.}      
    \end{cases}
\end{equation}
Despite the highly irregular source term, and the inevitable discontinuities of $u$ across the boundary, one can obtain several useful inequalities. Crucially, we obtain the second-order estimates
\begin{equation} \label{2 order intro}
u_t(x,t)\leq C_R, \quad u_{xx}(x,t)\chi_{x>\Lambda_t} \leq C_R \quad  (x,t)\in (-\infty,R) \times (R^{-1},\infty),  
\end{equation}
valid for any probabilistic solution (see Proposition \ref{prop: semiconvexity} and Lemma \ref{lem: uxx ux bd}), as well as the linear nondegeneracy estimate
\begin{equation} \label{nondegeneracy intro}
u(x,t)\geq c_R (x-\Lambda_t)^+, \quad (x,t)\in (-\infty,R) \times (R^{-1},R), 
\end{equation}
valid for any physical solution, any $(x_0,t_0)\in \R \times (0,\infty)$, and a small constant $c_R>0$ (see Proposition \ref{prop: u nondegeneracy}). 

Estimates \eqref{2 order intro} and \eqref{nondegeneracy intro} are key in our analysis. We remark, for instance, that letting $x \downarrow \Lambda_{t}$ in \eqref{nondegeneracy intro}
formally yields a positive lower bound on the velocity $\dot \Lambda_t$, which is equivalent to a Lipschitz estimate on the freezing time $s=\Lambda^{-1}$. On the other hand, the one-sided bounds of \eqref{2 order intro} allow us to interpret quantities such as $u(\Lambda_{t},t-)$ or $u_x(\Lambda_t,t)$ in the pointwise sense, and, combined with \eqref{nondegeneracy intro}, make it possible to prove Theorem \ref{thm: local to global uniqueness} by ruling out oscillatory behavior for $u$ near the boundary.

The time derivative estimate of \eqref{2 order intro} is inspired by a similar, new bound for the classical Stefan problem, which was recently obtained by A. Figalli, X. Ros-Oton, and J. Serra in \cite[Prop. 3.4]{figalli}. We note that our derivation of this bound does not rely on the one-dimensional setting, and does not need $u$ to be continuous.  In the context of the supercooled Stefan problem, the weak formulation of \eqref{u eq intro}, for general target measures $\nu$, was first developed in \cite{KK,CKK} to obtain existence of weak solutions in arbitrary dimensions (see also \cite{GhoKimPal}).

In Section \ref{sec: jumps}, we show that singular times, including jump times, cannot accumulate. We treat all singular times in a unified way, with non-jump singularities playing the role of degenerate, zero-length jumps. The first ingredient is a backward propagation of oscillation property. That is, we prove that, given two time levels $t_*>t_0>0$, any sign changes of $u_x$ at time $t_*$ can be transported backward along continuous curves to corresponding sign changes of $u_x(\cdot,t_0)$ (see Lemma \ref{lem:global curves}). This part of our proof is a refined version of the construction in \cite[Lem.\ 4.1]{delarue}. The main improvement is that we ensure that our curves can be defined globally on $[t_0,t_*]$ without ever hitting the boundary or escaping to $+\infty$. This allows us to quantify the oscillations at $t_*$ purely in terms of the oscillations at $t_0$, and to use the globally defined curves as barriers for a recursive argument.

The second ingredient is the fact that each  singular time $t_*>t_0$ forces one additional oscillation of $u$ near the boundary at time $t_*$ (see Lemma \ref{lem: jump implies oscillation}). These oscillations can then be propagated back to the fixed time $t_0$, which allows us to estimate the number of future such times in terms of the number of past oscillations at $t_0$. Since we showed that the latter is locally finite, this allows us to rule out their accumulation, and in particular to prove Theorem \ref{thm: jumps}.

In Section \ref{sec:regularity}, we prove the main regularity results for the solution and its free boundary. For this purpose, we utilize the potential
\begin{equation} \label{intro duvaut}
    w(x,t)=\int_{t}^{\infty}u(x,s)ds,
\end{equation}
which satisfies the parabolic obstacle problem
\begin{equation} \label{intro obstacle} w_t-\frac{1}{2}w_{xx}=-\nu(x)\chi_{w>0}.    
\end{equation}
Taking advantage of the local finiteness of jump times, one sees that, near any free boundary point corresponding to $x_0\notin\overline{J}$, $\nu\equiv \alpha^{-1}$. In that case, \eqref{intro obstacle} becomes the standard obstacle problem with constant source, which allows us to perform a classical blow-up analysis on these non-jump points (see \cite{CafObstacle,CafPetSha}), to achieve a precise understanding of the singular points, and to prove Theorem \ref{thm: classification}. Finally, we leverage the information obtained from the blow-up analysis, as well as the second-order estimates \eqref{2 order intro}, to show that the freezing time is of class $C^1$, proving Theorem \ref{thm: C1 regu}.
 
We note that time-integrated transformations of the form \eqref{intro duvaut} are classical in Stefan-type problems. For the classical Stefan problem, this goes back to \cite{duvaut}; in the supercooled setting, related formulations appear, for instance, in \cite{DiBenedettoFriedman,FasanoPrimicerioHowisonOckendon,HerreroVelazquez}. In the more recent higher-dimensional probabilistic formulation, the potential $w$ was  used in \cite{KK} to establish, among other things, the lower-semicontinuity of the freezing time $s$, and \cite[Sec. 4.3]{KK} suggested obtaining higher regularity of the free boundary through the regularity theory of the parabolic obstacle problem for $w$. 

Since this paper was first posted, the local free boundary regularity theory for the supercooled Stefan problem in $\R^d$ was developed in \cite{EKM}, where the $C^1$ regularity of $s$ and the structure of the singular set are established in arbitrary dimensions. The companion paper \cite{CKM} studies, also in arbitrary dimensions, \emph{maximal solutions} \cite{CKK,KK}, a selection principle based on delaying solidification as much as possible and playing a role analogous, in the compactly supported setting, to the minimal-jump selection in \eqref{eq: physical defi}; it also establishes further regularity properties and stability results for this class.
\section{Analysis of probabilistic solutions and their associated obstacle problem} \label{sec:estimates}
Throughout the paper, we will tacitly assume that $u_0$ is a probability density on $[0,\infty)$. We refer to a probabilistic solution $\Lambda$ as trivial if  $\Lambda_0=\alpha$, so that all freezing happens at the initial time, and $u \equiv 0$. This occurs, for instance, when $u_0=\alpha^{-1}\mathbf 1_{(0,\alpha)}$ . Since the trivial case satisfies all of the main theorems immediately, we also tacitly assume that all probabilistic solutions treated below are nontrivial. The results of this section are valid for all such probabilistic solutions $\Lambda$, as they do not make use of the physical jump condition \eqref{eq: physical defi}.
\subsection{The equation satisfied by $u$}
We begin with a basic estimate on $u$ and its time integral, which, in particular, shows that $u$ is well-defined as a density.
\begin{lem} \label{lem: u bound} Let $\Lambda$ be a probabilistic solution to \eqref{supercooled eqs}, and let $u$ be given by \eqref{u defi}. Then, after redefining $u$ on a set of measure zero, $u$ is smooth in the open set
\[\Omega:=\{(x,t)\in \R\times (0,\infty): x>\Lambda_t\},\]
vanishes on $\{(x,t):x\leq\Lambda_t\}$, satisfies $\Omega=\{u>0\}$, and solves
\[u_t-\frac{1}{2}u_{xx}=0, \quad(x,t)\in \Omega.\]
 Moreover, there exists a universal constant $C>0$ such that
    \begin{equation} \label{u bound}u(x,t) \leq Ct^{-1/2},  \quad (x,t)\in \R \times (0,\infty),        
    \end{equation}
    and
    \begin{equation} \label{w bound}
        \int_{0}^{\infty}u(x,s)ds \leq 2x_+,\quad x\in \R.
    \end{equation}
\end{lem}
\begin{proof}  Letting $\tau_0=\inf\{t:B_t\leq 0\}$, we note that $\tau\leq \tau_0$, because $\Lambda\geq0$. Hence, for any $[a,b]\subset [0,\infty)$, 
\begin{equation} \label{eq:stopped-domination-by-killed-heat}
   \P(B_t\in [a,b], t< \tau )\leq \P(B_t \in [a,b], t<\tau_0)=\int_{a}^bv(x,t)dx,
\end{equation}
where $v$ is the solution to
\[v_{t}-\frac12v_{xx}=0, \quad v(x,0)=u_0(x), \quad v(0,t)=0, \quad (x,t)\in (0,\infty) \times (0,\infty).\]
By \eqref{eq:stopped-domination-by-killed-heat}, the measure $\P(B_t\in dx,t<\tau)$ is dominated by $v(x,t)dx$ on $[0,\infty)$. Thus $u$ is well-defined as a density. Since $t<\tau$ implies $B_t>\Lambda_t$, we choose the density $u$ in \eqref{u defi} to vanish whenever $x\leq\Lambda_t$. We then have
\begin{equation} \label{uineqdokap}
    u(x,t)\leq v(x,t)= \int_{0}^{\infty}(G(x-y,t)-G(x+y,t))u_0(y)dy, \quad \text{for a.e. } x>0, \quad  t>0,
\end{equation}
where
\begin{equation}
    G(x,t)=(2\pi t)^{-1/2}\exp (-x^2/{2t}).
\end{equation}
In particular, this gives \eqref{u bound} for a.e. $(x,t)$. Furthermore, by a direct computation, \eqref{uineqdokap} yields
\begin{multline}
\int_0^{\infty}u(x,t)dt \leq \int_0^{\infty}\int_{0}^{\infty}(G(x-y,t)-G(x+y,t))u_0(y)dy dt=2\int_{0}^{\infty} \min(x,y)u_0(y)dy\leq 2x.\end{multline}   
which shows \eqref{w bound}. Note that, since $\Lambda$ is nondecreasing and right continuous, it is upper semicontinuous. Therefore, the set $\Omega$ is open. Letting $\varphi \in C^{\infty}_c(\Omega)$, and noting that $\varphi(B_{\tau},\tau)=\varphi(B_0,0) \equiv 0$, we have, by It\^o's formula and the compact support of $\varphi$ (see \cite[Ch. 3, Cor. 3.6; Ch. 4, Thm. 3.3]{RevYor}),
\begin{multline}\label{qwodkr8}
    0=\E(\varphi(B_{\tau},\tau))= \E(\varphi(B_0,0))+\E\left( \int_{0}^\tau \left(\frac{1}{2}\varphi_{xx}(B_s,s)+\varphi_t(B_s,s)\right)ds\right)
    \\=\E\left( \int_{0}^{\infty} \left(\frac{1}{2}\varphi_{xx}(B_s,s)+\varphi_t(B_s,s)\right)\chi_{s<\tau}ds\right)= \int_{0}^{\infty}\E\left(  \left(\frac{1}{2}\varphi_{xx}(B_s,s)+\varphi_t(B_s,s)\right)\chi_{\{s<\tau\}}\right)ds\\=\int_{0}^{\infty}\int_{\R} \left( \frac{1}{2}\varphi_{xx}+\varphi_t\right)u(x,s)dxds,\end{multline}
where \eqref{u defi} was used in the last equality. This shows that $u$ solves the heat equation in the distributional sense in $\Omega$. By interior regularity, after redefining $u$ on a null set in $\Omega$, it is smooth and solves the equation classically there. The a.e. comparison in \eqref{uineqdokap} then holds pointwise by continuity, so \eqref{u bound} holds pointwise as stated. By the preceding convention, $u$ vanishes outside of $\Omega$. Hence, by the strong maximum principle, $\Omega=\{u>0\}$.
\end{proof}
The next lemma concerns the regularity of the past limit $u(x,t_0-)$, for $x>\Lambda_{t_0-}$, which will be of particular importance when $\Lambda_{t_0}>\Lambda_{t_0-}$. It also identifies $u(x,t_0-)dx$ as a concrete sub-probability measure. The limit at $x=\Lambda_{t_0-}$ is more delicate, and will be treated later (see Lemma \ref{lem: lim past regu}).
\begin{lem} \label{lem: smooth ext} Let $\Lambda$ be a probabilistic solution to \eqref{supercooled eqs}, let $u$ be given by \eqref{u defi}, and let $t_0>0$. Then for $x\in (\Lambda_{t_0-},\infty)$, the limit
\[ u(x,t_0-):=\lim_{t\uparrow t_0}u(x,t)\]
exists. The function $u(\cdot, t_0-)$ is analytic in space, and defines a strictly positive, smooth extension of $u$ in $(\Lambda_{t_0-},\infty)\times (0,t_0]$. Moreover, we have
\begin{equation} \label{u(t-) formula}
    \P(B_{t_0}\in [a,b], t_0\leq \tau) =\int_{a}^bu(x,t_0-)dx,  \quad \Lambda_{t_0-} \leq a <b.
\end{equation}
\end{lem}
\begin{proof} By Lemma \ref{lem: u bound}, for any $\delta\in(0,1/2)$, the function $u$ is bounded and satisfies the heat equation in the cylinder $[\Lambda_{t_0-}+\delta,\Lambda_{t_0-}+\delta^{-1}] \times (t_0/4,t_0)$. Hence, by the interior Schauder estimates (see \cite[Thm. 4.9]{LiebermanBook}), for any $k>1$, $u$ is uniformly bounded in $C^{k}([\Lambda_{t_0-}+2\delta,\Lambda_{t_0-}+(2\delta)^{-1}] \times (t_0/2,t_0))$, which implies the smooth extension claim.

    To justify the analyticity at $t=t_0$, fix a compact interval $K\Subset(\Lambda_{t_0-},\infty)$. By the smooth extension just obtained, the trace $\phi:=u(\cdot,t_0-)$ is smooth in a neighborhood of $K$. After extending $\phi$ from this neighborhood to a compactly supported smooth function on $\R$ and solving the heat equation forward from time $t_0$, we may glue this forward solution to $u$ across $t_0$ in a smaller neighborhood of $K\times\{t_0\}$. Thus $u(\cdot,t_0-)$ is locally the time slice of a classical heat-equation solution at an interior time, and \cite[Ch. VI, Thm. 1]{mikhailov} implies that it is analytic on $K$. Since $K$ was arbitrary, $u(\cdot,t_0-)$ is analytic in $(\Lambda_{t_0-},\infty)$. The strict positivity follows from the strong maximum principle for the heat equation. Finally, letting $t\uparrow t_0$ in \eqref{u defi}, and using the local boundedness of $u$, we obtain
\[\int_{a}^bu(x,t_0-)dx=\lim_{t\uparrow t_0} \P(B_t\in [a,b], t< \tau)=\lim_{t\uparrow t_0} \P(B_{t_0}\in [a,b], t<\tau)=\P(B_{t_0}\in [a,b], t_0 \leq \tau),\]
where the continuity (in probability) of Brownian motion was used in the second equality, and the monotone convergence theorem was used in the third equality.
\end{proof}
Throughout the rest of the paper, we take \eqref{nu defi} as the definition of $\nu$. A straightforward computation shows that this formula agrees with the probabilistic interpretation $B_{\tau}\sim \nu(x)dx$, but we will not make use of this precise characterization. For our purposes, it is enough to show below that $\nu$ has total mass one, and that the freezing time $s$ is continuous.
\begin{lem}\label{lem: lambda strict incr}Let $\Lambda$ be a probabilistic solution to \eqref{supercooled eqs}, let $s$ be given by \eqref{s defi}, and let $\nu$ be given by \eqref{nu defi}. Then $\Lambda$ is strictly increasing, $s$ is continuous, and
\begin{equation}
    \int_{\R}\nu(x)dx=1.
\end{equation}
\end{lem}
\begin{proof} We have, for $h>0$, recalling \eqref{u defi}, the monotonicity of $\Lambda$, and the Markov property of Brownian motion,
 \begin{multline}
     \alpha^{-1}(\Lambda_{t+h}-\Lambda_t)=\P(t<\tau \leq t+h)=\P\left(\inf_{s\in(0,h]}(B_{t+s}-\Lambda_{t+s})\leq 0, t<\tau \right)\\
     \geq\P\left(\inf_{s\in(0,h]}B_{t+s}\leq \Lambda_{t}, t<\tau\right)=\int_{\Lambda_t}^{\infty} \P \left(\inf_{s\in(0,h]}B_s\leq \Lambda_t \Bigr \vert B_{0}=x\right) u(x,t)dx>0,
 \end{multline}   
 where in the last inequality we used the strict positivity of $u$ in $\{\Lambda_t<x\}$ and the standard reflection principle for Brownian motion (see \cite[Exm. 8.4.1]{durrett}). The continuity of $s(x)=\inf\{t: x<\Lambda_t\}$ then follows from the strict monotonicity of $\Lambda$. 

Now, letting $A=\{t\in [0,\infty): \Lambda_{t}=\Lambda_{t-}\},$  observe that $s^{-1}(A)=\{x \in [0,\alpha): \Lambda_{s(x)-}=\Lambda_{s(x)}\}$, and $\Lambda_{s(x)}=x$ for all $x\in s^{-1}(A)$. Therefore, by change of variables (see \cite[Prop. 4.9]{RevYor}), we have
  \begin{equation} \label{cov nu231}
    \int_{\Lambda_{s(x)}=\Lambda_{s(x)-}}\nu(x)dx=\alpha^{-1}\int_{\Lambda_{s(x)}=\Lambda_{s(x)-}}dx=\alpha^{-1}\int_{\Lambda_{t}=\Lambda_{t-}}d\Lambda_t.
    \end{equation}
On the other hand, noting that $\Lambda$ has countably many discontinuities, we obtain from \eqref{u(t-) formula} that    
 \begin{multline}\label{cov nu2312}
   \int_{\Lambda_{s(x)}>\Lambda_{s(x)-}}\nu(x)dx=\sum_{\Lambda_{t}>\Lambda_{t-}} \int_{\Lambda_{t-}}^{\Lambda_t}u(x,t-)dx=\sum_{\Lambda_{t}>\Lambda_{t-}} \P(B_t\in [\Lambda_{t-},\Lambda_t],\tau\geq t)=\sum_{\Lambda_{t}>\Lambda_{t-}} \P(\tau=t)\\=\sum_{\Lambda_{t}>\Lambda_{t-}}\alpha^{-1}(\Lambda_{t}-\Lambda_{t-})=\alpha^{-1}\int_{\Lambda_t>\Lambda_{t-}} d\Lambda_t.
\end{multline}   
Adding \eqref{cov nu231} and \eqref{cov nu2312}, we conclude that
\begin{equation}    \int_{\R}\nu(x)dx=\alpha^{-1}\int_{0}^{\infty}d\Lambda_t=1.
    \end{equation}
\end{proof}
We now derive the equation solved by $u$ in the whole space. We say that a locally integrable function $v$ is subcaloric in an open set $Q\subset \R\times(0,\infty)$ if $v_t-\frac12v_{xx}\leq0$ in the sense of distributions in $Q$.
\begin{prop} \label{prop: u eq} Let $\Lambda$ be a probabilistic solution to \eqref{supercooled eqs}, and let $u$ be given by \eqref{u defi}. The function $u$ satisfies, in the distributional sense,
    \begin{equation}\label{u equation}\begin{cases} u_t-\frac12 u_{xx}=(\nu(x)\chi_{u>0})_t & x\in \R \times (0,\infty), \\
    u(x,0)=u_0(x) \chi_{x>\Lambda_0} &x\in \R.
    \end{cases}
\end{equation}
Namely, for every $\varphi\in C_c^{\infty}(\R \times [0,\infty))$,
\begin{equation}\label{u equation dist}
     \int_{0}^{\infty}\int_{\R} \left( \varphi_t(u-\nu \chi_{u>0})+\frac{1}{2}\varphi_{xx}u\right)dxdt+\int_{\R}(u_0(x)-\nu(x))\varphi(x,0)dx=0.
\end{equation}
In particular, $u$ is subcaloric in $\R \times (0,\infty)$, and $u \in L^2_{\operatorname{loc}}((0,\infty);H^1_{\operatorname{loc}}(\R))$.
\end{prop}
\begin{proof} Let $\varphi\in C_c^{\infty}(\R \times [0,\infty))$. Using It\^o's formula as in \eqref{qwodkr8},
\begin{multline}\label{qwodkr8u392qwxk12xv}
    \E(\varphi(B_{\tau},\tau))= \E(\varphi(B_0,0))+\E\left( \int_{0}^\tau \left(\frac{1}{2}\varphi_{xx}(B_s,s)+\varphi_t(B_s,s)\right)ds\right)
    \\=\int_{\R}u_0(x)\varphi(x,0)dx+ \int_{0}^{\infty}\int_{\R} \left( \frac{1}{2}\varphi_{xx}+\varphi_t\right)u(x,t)dxdt.\end{multline}
    Since $B_{\tau}=\Lambda_{\tau}$ whenever $\Lambda_{\tau}=\Lambda_{\tau-}$, we have
\begin{equation}  \label{expcqwsum02s}\E(\varphi(B_{\tau},\tau))=\E(\varphi(\Lambda_{\tau},\tau)\chi_{\{\Lambda_{\tau}=\Lambda_{\tau-}\}})+\E(\varphi(B_{\tau},\tau)\chi_{\{\Lambda_{\tau}>\Lambda_{\tau-}\}}).\end{equation}
Now, observing that $\alpha^{-1}\Lambda_t=\P(\tau \leq t)$ is the cumulative distribution function of $\tau$, it follows that
    \begin{equation} \label{qwodkpsxk12xv}
        \E(\varphi(\Lambda_{\tau},\tau)\chi_{\{\Lambda_{\tau}=\Lambda_{\tau-}\}})=\int_{\Lambda_s=\Lambda_{s-}}\varphi(\Lambda_s,s)d(\alpha^{-1}\Lambda_s).
    \end{equation}
Therefore, changing variables as in \eqref{cov nu231}, we have
     \begin{multline} \label{0392kcmew0mc0m}
         \int_{\Lambda_s=\Lambda_{s-}}\varphi(\Lambda_s,s)d(\alpha^{-1}\Lambda_s)
         =\alpha^{-1}\int_{{\Lambda_{s(x)}=\Lambda_{s(x)-}}}\varphi( x,s(x))dx\\=\alpha^{-1}\int_{\Lambda_{s(x)}=\Lambda_{s(x)-}}\int_{0}^{s(x)}\varphi_tdtdx+\alpha^{-1}\int_{\Lambda_{s(x)}=\Lambda_{s(x)-}}\varphi(x,0)dx\\=\int_{\Lambda_{s(x)}=\Lambda_{s(x)-}}\int_{0}^{\infty}\varphi_t\nu\chi_{u>0}dtdx+\int_{\Lambda_{s(x)}=\Lambda_{s(x)-}}\varphi(x,0)\nu(x)dx,
     \end{multline}
where the integrals are understood to be over $x\in (0,\alpha)$. Thus, we infer from \eqref{qwodkpsxk12xv} and \eqref{0392kcmew0mc0m} that
\begin{equation} \label{epx1dwlpas}
    \E(\varphi(\Lambda_{\tau},\tau)\chi_{\{\Lambda_{\tau}=\Lambda_{\tau-}\}})=\int_{\Lambda_{s(x)}=\Lambda_{s(x)-}}\int_{0}^{\infty}\varphi_t\nu\chi_{u>0}dtdx+\int_{\Lambda_{s(x)}=\Lambda_{s(x)-}}\varphi(x,0)\nu(x)dx.
\end{equation}
On the other hand, since the set of discontinuities of $\Lambda$ is countable, using \eqref{u(t-) formula} and recalling the convention $u(\cdot,0-):=u_0$, we have
\begin{multline}  \label{qw2dklcfpqws2}   \E(\varphi(B_{\tau},\tau)\chi_{\{\Lambda_{\tau}>\Lambda_{\tau-}\}})
=\sum_{\Lambda_{t}>\Lambda_{t-}}\E(\varphi(B_{\tau},\tau)\chi_{\{\tau=t\}})
=\sum_{\Lambda_{t}>\Lambda_{t-}}\E(\chi_{B_t \in [\Lambda_{t-},\Lambda_t]}\varphi(B_{t},t)\chi_{\{\tau\geq t\}})\\=\sum_{\Lambda_{t}>\Lambda_{t-}}\int_{\Lambda_{t-}}^{\Lambda_{t}}\varphi(x,t)u(x,t-)dx
=\sum_{\Lambda_{t}>\Lambda_{t-}}\int_{\Lambda_{t-}}^{\Lambda_{t}}\varphi(x,t)\nu(x)dx
\\=\sum_{\Lambda_{t}>\Lambda_{t-}}\int_{\Lambda_{t-}}^{\Lambda_{t}}\int_{0}^{\infty}\varphi_t\nu\chi_{u>0}dtdx+\int_{\Lambda_{t-}}^{\Lambda_t}\varphi(x,0)\nu(x)dx\\
=\int_{\Lambda_{s(x)}>\Lambda_{s(x)-}}\int_0^{\infty}\varphi_t\nu \chi_{u>0}dtdx+\int_{\Lambda_{s(x)}>\Lambda_{s(x)-}}\varphi(x,0)\nu(x)dx.
\end{multline}
In view of \eqref{expcqwsum02s}, \eqref{epx1dwlpas}, and \eqref{qw2dklcfpqws2}, using the fact that $\varphi$ is compactly supported and $\nu$ is supported in $[0,\alpha]$, we have
\begin{equation}
\E(\varphi(B_{\tau},\tau))=\int_{0}^{\infty}\int_{\R}\varphi_t(x,t)\nu(x)\chi_{u>0}dxdt+\int_{\R}\varphi(x,0)\nu(x)dx. 
\end{equation}
The equality \eqref{u equation dist} thus follows from \eqref{qwodkr8u392qwxk12xv}. Finally, the fact that $u$ is subcaloric follows readily from the fact that the set $\{u(t)>0\}$ is decreasing in time. Namely,
\[u_t-\frac12u_{xx}=(\nu\chi_{u>0})_t\leq0,\]
and the $L^2_{\operatorname{loc}}((0,\infty);H^1_{\operatorname{loc}}(\R))$ bounds on $u$ then follow from this inequality by standard energy estimates.
\end{proof}
\begin{rem}
The source term in \eqref{u equation} is a compact way of writing the contribution of the moving interface in weak form. Indeed, after moving it to the right-hand side in \eqref{u equation dist}, this contribution can be written as
\begin{equation}
\begin{aligned}
\int_{0}^{\alpha}\nu(x)\varphi(x,s(x))dx
&=\alpha^{-1}\int_{0}^{\infty}\varphi(\Lambda_{t-},t)d\Lambda_t\\
&\quad+\sum_{\Lambda_t>\Lambda_{t-}}\int_{\Lambda_{t-}}^{\Lambda_t}u(x,t-)\bigl(\varphi(x,t)-\varphi(\Lambda_{t-},t)\bigr)dx,
\end{aligned}
\end{equation}
which follows from \eqref{nu defi}, \eqref{u(t-) formula}, and $\int_{\Lambda_{t-}}^{\Lambda_t}u(x,t-)dx=\alpha^{-1}(\Lambda_t-\Lambda_{t-})$ at jump times (cf. \cite[Eq.~(2.9)]{LedSojTransport}).
\end{rem}
\subsection{The weighted obstacle problem} 
In view of Lemma \ref{lem: u bound}, the function
\begin{equation} \label{w defi}
    w(x,t)=\int_{t}^{\infty}u(x,s)ds,\quad (x,t)\in \R \times (0,\infty),
\end{equation}
is well defined. We show now that $w$ is locally $C^{1,1}$ in the spatial variable and locally Lipschitz in time, and solves a parabolic obstacle problem with the weight $\nu$. The properties of $w$ will be exploited in subsequent sections to obtain the key estimates for $u$ and to analyze the free boundary.

\begin{prop} \label{prop: w eq} Let $\Lambda$ be a probabilistic solution to \eqref{supercooled eqs}, let $u$ be given by \eqref{u defi}, and let $w$ be given by \eqref{w defi}.  The functions $w_t$ and $w_{xx}$ are locally bounded, and the function $w_x$ is locally H\"older continuous. In fact, for every $\beta \in (0,1)$, and every $R\geq 1$, there exists a constant $C_R$, depending only on $R$, $\alpha$, and $\beta$, such that
\begin{multline}\|w_t\|_{L^{\infty}(\R\times (R^{-1},\infty))}+\|w_{xx}\|_{L^{\infty}(\R \times (R^{-1},\infty))}+\|w\|_{L^{\infty}((-\infty,R)\times(0,\infty))}\\+\|w_x\|_{C^{\beta,\beta/2}((-\infty,R)\times(R^{-1},\infty))}\leq C_R.\end{multline}
Furthermore, $w$ satisfies
\begin{equation} \label{w equation}
   \begin{cases} w_t-\frac{1}{2}w_{xx}=-\nu(x)\chi_{w(x,t)>0}\quad (x,t)\in \R\times (0,\infty),\\
   w\geq 0,\,\, w_t\leq0, \,\,\{w>0\}=\{w_t<0\}.
   \end{cases}
\end{equation}

\end{prop}
\begin{proof} The fact that $w_t \leq 0$ and $\{w>0\}=\{u>0\}$ follows directly from the definition. Moreover, we infer from \eqref{u bound} that, for $R \geq 1$, there exists a constant $C>0$, depending only on $\alpha$ and $\beta$, such that
\begin{equation}
 w(x,t)\leq w(x,0)\leq 2x\leq 2R, \;\;x\leq R,\;\; t\in [0,\infty).  
\end{equation}
To see that $w$ satisfies \eqref{w equation}, we proceed formally first. Integrating \eqref{u equation} from $t$ to $T$, and using the fact that $w_t=-u$ and $\nu$ is supported on $(0,\alpha)$, we have
\begin{equation}
-\nu(x)\chi_{w>0}(x,t)=\nu(x)\chi_{w>0}(x,s)\Bigr\rvert_{s=t}^{s=\infty} =-\left(w_t-\frac12w_{xx}\right) (x,s)\Biggr \rvert_{s=t}^{s=\infty}=\left(w_t-\frac12w_{xx}\right)(x,t). 
\end{equation}
The precise justification of this formal computation proceeds as follows. Let $\psi\in C^{\infty}_c(\R \times(0,\infty))$, let $R\geq 1$, and let 
\[ \varphi(x,t)= \zeta_R(t)\Psi(x,t), \quad \Psi(x,t)=\int_0^t\psi(x,s)ds, \quad  \zeta_R(t)=\zeta(t/R),\] where $\zeta\in C^{\infty}_c([0,\infty))$ satisfies $\zeta\geq0$, $\zeta \equiv 1$  on $[0,1]$, and $\zeta \equiv 0$ on $[2,\infty)$. Then, noting that $\chi_{w>0}=\chi_{u>0}$, and
$\zeta_R\psi \equiv \psi$ for sufficiently large $R$, \eqref{u equation dist} yields
\begin{equation} \label{opsequjf1}
     \int_{0}^{\infty}\int_{\R} \left( \psi u-\varphi_t\nu \chi_{w>0}+\zeta_R'\Psi u+\frac{1}{2}\varphi_{xx}u\right)dxdt=0.
\end{equation}
Since $\varphi$ is compactly supported and $\nu(x)\chi_{w>0}(x,\cdot)\equiv 0$ for $x\notin (\Lambda_0,\alpha)$, we have, for sufficiently large $R$,
\begin{equation} \label{opsequjf2}
    \int_{\R} \int_{0}^{\infty} \varphi_t\nu \chi_{w>0}dtdx=\int_{\Lambda_0}^{\alpha}  \int_{0}^{\infty} \varphi_t\nu \chi_{w>0}dtdx=\int_\R  \int_{0}^{\infty} (\psi\nu \chi_{w>0})dtdx+\int_{\Lambda_0}^{\alpha}\int_{0}^{\infty} \zeta_R'\Psi\nu \chi_{w>0}dtdx,
\end{equation}
and
\begin{equation}\label{opsequjf3}
\left|\int_{\Lambda_0}^{\alpha}\int_{0}^{\infty} \zeta_R'\Psi\nu \chi_{w>0} dtdx\right|=\left| \int_{\Lambda_0}^{\alpha}R^{-1}\int_{R}^{2R}\zeta'(t/R)\Psi(x,t)\nu(x)\chi_{w>0} dtdx \right|\leq C 
\left|\int_{\Lambda_0}^{\alpha} \nu(x)\chi_{R\leq s(x)}dx\right|:=CE_R,
\end{equation}
where $C>0$ depends only on $\zeta$, $\psi$, and $\alpha^{-1}$. We observe that the integrand in $E_R$ equals zero whenever $R>s(x)$, and is bounded by the integrable function $\nu$. Therefore, by dominated convergence, $E_R=o(1)$ as $R\to \infty$. Next, we have, for sufficiently large $R$,
\begin{multline}\label{opsequjf4}
\int_{\R}\int_{0}^{\infty}\varphi_{xx}udtdx=-\int_{\R}\int_{0}^{\infty}\varphi_{xx}w_tdtdx\\=\int_{\R}\int_0^{\infty}\psi_{xx}wdtdx +\int_{\R}\int_0^{\infty}\zeta_R'\Psi_{xx}wdtdx -\left(\int_{\R} \varphi_{xx}w(x,t)dx\right) \Biggr \rvert_{t=0}^{t=\infty}=\int_{\R}\int_0^{\infty}\psi_{xx}wdxdt+o(1)
\end{multline}
as $R\to \infty$, where the term $\left|\int_{\R}\int_0^{\infty}\zeta_R'\Psi_{xx}wdtdx \right|$ was estimated in the same way as \eqref{opsequjf3}.
Finally, if $A$ is the spatial projection of the support of $\varphi$, then \eqref{u bound} yields
\begin{equation}\label{opsequjf5}
    \left| \int_{\R} \int_{0}^{\infty} \zeta'_R\Psi u dtdx\right| \leq C \left| \int_{A} \sup_{t\in[R,2R]}u(x,t)dx \right|\leq \frac{C}{\sqrt{R}} \left| A \right| =o(1)
\end{equation}
as $R\to \infty$. Hence, putting together \eqref{opsequjf1}, \eqref{opsequjf2}, \eqref{opsequjf3}, \eqref{opsequjf4}, and \eqref{opsequjf5}, we obtain
\begin{equation}
     \int_{0}^{\infty}\int_{\R} \left( -\psi w_t-\psi\nu \chi_{w>0}+\frac{1}{2}\psi_{xx}w\right)dxdt=o(1),
\end{equation}
which yields \eqref{w equation} by letting $R \to \infty$. The local bound on $|w_t|=u$ was shown in Lemma \ref{lem: u bound}, and  the bound on $|w_{xx}|$ then follows directly from \eqref{w equation} and \eqref{nu defi}. Finally, the $C^{\beta,\beta/2}$ estimate for $w_x$ follows from the parabolic Sobolev embedding (see \cite[Ch. II, Lem. 3.3]{ladyzhenskaya}).
\end{proof}
\subsection{Second-order estimates for $u$}
We now derive the second-order bounds on the solution that will allow us to prove the nondegeneracy and regularity results. Given $(x_0,t_0) \in \R\times (0,\infty)$ and $r>0$, we define the parabolic cylinders
\begin{equation}
    Q_r(x_0,t_0):=(x_0-r,x_0+r)\times (t_0-r^2,t_0+r^2), \quad Q_r^-(x_0,t_0):=(x_0-r,x_0+r)\times (t_0-r^2,t_0].
\end{equation}
We begin with an upper bound on the time derivative.
\begin{prop}\label{prop: semiconvexity} Let $\Lambda$ be a probabilistic solution to \eqref{supercooled eqs}, let $u$ be given by \eqref{u defi}, and let $w$ be given by \eqref{w defi}.  Then $u_t$ is locally bounded from above in $\R \times (0,\infty)$. Namely, for every $R\geq1$, there exists a constant $C_R$, depending only on $R$, such that, in the sense of distributions,
    \begin{equation}
    u_t \leq C_R \quad \text{in } (-\infty,R)\times(R^{-2},\infty).
\end{equation}
\end{prop}
\begin{proof} Let $(x_0,t_0) \in \R \times (0,\infty)$,  $r=\frac{\sqrt{t_0}}{10}$, $D=Q_{6r}(x_0,t_0)$, $h\in (0,r^2)$, and let
\[z(x,t)=\frac{u(x,t+h)-u(x,t)}{h}.\]
In $\{u>0\}$, we have
\[z_t-\frac12 z_{xx}=  h^{-1}(\nu \chi_{u(t+h)>0})_t \leq0,\]
so $z$ is subcaloric in $\{u>0\}\cap D$. This implies, in particular, that $z^+$ is subcaloric in $\{u>0\}\cap D$. Given that $s$ is continuous, the set
\[ A=\{(x,t)\in D: s(x)<t+h\}=\{(x,t)\in D: u(x,t+h)=0\}^{\circ}\] is an open neighborhood of $\{u=0\}\cap D$, disjoint from $\{u(t+h)>0\}$. Additionally, for $(x,t)\in A$, $z^+(x,t)=h^{-1}(u(x,t+h)-u(x,t))^+=h^{-1}(-u(x,t))^+=0$, so trivially $z^+$ is subcaloric in $A$. Noting that $D$ is the union of the overlapping open sets $A \cup (\{u>0\}\cap D)=D$, in both of which $z^+$ is subcaloric, we conclude that $z^+$ is subcaloric in $D$. Therefore, by the parabolic Krylov--Safonov $L^{\vep}$ estimate \cite[Thm. 4.16]{wang}, writing $Q_r := Q_r(x_0,t_0)$, and fixing some arbitrary $\vep \in (0,1)$, we have
\begin{equation} \label{cz2-edlsmg1}
\|z^+\|_{L^{\infty}(Q_r)}\leq C \left( \fint_{Q_{2r}} (z^+)^{\vep}  \right)^{\frac{1}{\vep}} \leq C|Q_{2r}|^{-1}\text{sup}_{\theta>0}\theta |\{|z|>\theta\}\cap Q_{2r}|.
\end{equation}
We note that in the last inequality we used the embedding $L^1$--weak $\hookrightarrow L^{\vep}$, which holds for $\vep \in (0,1)$ by a standard interpolation argument. But $z = -v_t$, where
\[v(x,t)=h^{-1}(w(x,t+h)-w(x,t)),\]
so, by the Calder\'on-Zygmund estimate for the heat equation (see \cite[Thm. 3.5]{figalli}, \cite{jones}),
\begin{equation}\label{cz2-edlsmg2}\text{sup}_{\theta>0}\theta |\{|z|>\theta\}\cap Q_{2r}|\leq \text{sup}_{\theta>0}\theta |\{|v_t|+|v_{xx}|>\theta\}\cap Q_{2r}| \leq C(r^{-2}\|v\|_{L^1(Q_{3r})}+\|v_t-v_{xx}/2\|_{L^1(Q_{3r})}).\end{equation}
On the other hand, since $\nu \geq0$ and $w_t \leq 0$, we have, by Proposition \ref{prop: w eq},
\[v\leq 0, \quad v_t- v_{xx}/2=-h^{-1}\nu(x)(\chi_{w(t+h)>0}-\chi_{w(t)>0})\geq 0.\]
By testing against a smooth, compactly supported bump function $\zeta$, supported in $Q_{4r}$, that equals $1$ in $Q_{3r}$, this implies
\begin{equation}\label{cz2-edlsmg3}\|v_t-v_{xx}/2\|_{L^1(Q_{3r})}\leq \left|\int(v_t-v_{xx}/2)\zeta\right|=\left|\int(\zeta_t+\zeta_{xx}/2)v\right| \leq \frac{C}{r^2}\|v\|_{L^1(Q_{4r})}.\end{equation}
Similarly, letting $\zeta$ be supported in $Q_{5r}$ and equal to $1$ in $Q_{4r}$,
\begin{multline}\label{cz2-edlsmg4}\|v\|_{L^1(Q_{4r})}\leq\left|\int v(x,t)\zeta(x,t)dxdt\right| =\left|h^{-1}\int(w(x,t+h)-w(x,t))\zeta(x,t)dxdt\right|\\=\left|h^{-1}\int(\zeta(x,t-h)-\zeta(x,t))w(x,t)dxdt\right| \leq \|\zeta_t\|_{L^{\infty}(Q_{5r})}\|w\|_{L^1(Q_{6r})}\leq \frac{C}{r^2}\|w\|_{L^1(Q_{6r})}.\end{multline}
Thus, we conclude from \eqref{cz2-edlsmg1}, \eqref{cz2-edlsmg2}, \eqref{cz2-edlsmg3}, \eqref{cz2-edlsmg4}, and \eqref{w bound} that
\[\|z^+\|_{L^{\infty}(Q_r)}\leq Cr^{-7}\|w\|_{L^1(Q_{6r})}\leq Cr^{-4}(|x_0|+r). \]
The result follows by letting $h \to 0$. 
\end{proof}
As an immediate application, we obtain semiconcavity and gradient estimates that hold up to the boundary. We note, however, that unlike Proposition \ref{prop: semiconvexity}, these bounds do not preclude $(u_{xx})^+$ from having a singular part concentrated on the frontier.
\begin{lem} \label{lem: uxx ux bd}  Let $\Lambda$ be a probabilistic solution to \eqref{supercooled eqs}, and let $u$ be given by \eqref{u defi}. Then, for every $R\geq 1$, there exists a constant $C_R>0$, depending only on $R$, such that
    \begin{equation}
        u_{xx}(x,t)\leq C_R, \quad  u_x \geq -C_R,\quad (x,t)\in ((-\infty,R) \times (R^{-2},\infty)) \cap\{u>0\}.
    \end{equation}
\end{lem}
\begin{proof} The upper bound on $u_{xx}$ follows immediately from Proposition \ref{prop: semiconvexity}, in view of the fact that $u_{xx}=2u_t$ in $\{u>0\}$. On the other hand, since $(x+1,t)\in \{u>0\}$ whenever $(x,t)\in \{u>0\}$, the mean value theorem implies that $u(x+1,t)-u(x,t)=u_x(y,t)$ for some $y\in (x,x+1)$, so that
\begin{equation*}
-u_x(x,t)=u_x(y,t)-u_x(x,t) +u(x,t)-u(x+1,t) \leq  \max_{z\in [x,x+1]} (u_{xx}(z,t))+u(x,t)\leq C_{R}.
\end{equation*}    
\end{proof}
As a second application of Proposition \ref{prop: semiconvexity}, we can now make sense of the limiting value $u(\Lambda_{t_0-},t_0-)$, for any $t_0>0$.

\begin{lem}\label{lem: lim past regu}  Let $\Lambda$ be a probabilistic solution to \eqref{supercooled eqs}, let $u$ be given by \eqref{u defi}, and let $t_0>0$. The limit
\begin{equation} \label{lim past}
  u(x,t_0-):=\lim_{t\uparrow t_0}u(x,t)
\end{equation}
exists for all $x\in \R$ and defines a continuous extension of $u$ in $[\Lambda_{t_0-},\infty)\times (0,t_0]$.
\end{lem}
\begin{proof} By Proposition \ref{prop: semiconvexity}, $u_t$ is locally bounded from above, so the limit \eqref{lim past} exists for all $x\in \R$. In view of Lemma \ref{lem: smooth ext}, it suffices now to prove that $u(\cdot,t_0-)$ is right continuous at $x=a:=\Lambda_{t_0-}$, and that the convergence in \eqref{lim past} holds locally uniformly in $x \in [\Lambda_{t_0-},\infty)$. In fact, by the upper bound on $u_t$, assuming the continuity of $u(\cdot,t_0-)$, the locally uniform convergence follows automatically by Dini's theorem. Hence, the problem is reduced to showing that \[\lim_{x\downarrow a}u(x,t_0-)=u(a,t_0-).\]
For this purpose, observe that, by Lemma \ref{lem: lambda strict incr}, $u(x,t_0-\rho)>0$ for $x\geq a$ and small $\rho>0$. Hence, by Lemma \ref{lem: uxx ux bd}, applied to $u(\cdot,t_0-\rho)$ and letting $\rho \downarrow 0$, we have, for $x\in (a,a+1)$ and some constant $C>0$ independent of $x$,
\begin{equation} \label{ffs32d0d1k-1}
    u(a,t_0-) \leq  u(x,t_0-)+C(x-a).
\end{equation}
On the other hand, by definition of $u(\cdot,t_0-)$, given $\vep\in (0,1)$, 
\begin{equation}\label{ffs32d0d1k-2}
    u(a,t_0-)\geq u(a,\ot)-\vep
\end{equation}
for any $\ot<t_0$ sufficiently close to $t_0$. Fixing one such $\ot \in(t_0-\vep,t_0)$, and noting that, by Lemma \ref{lem: lambda strict incr}, $u(a,\ot)>0$, the smoothness of $u$ in $\{u>0\}$ implies that there exists $\delta\in (0,\vep)$, depending on $\ot$ and $\vep$, such that
\begin{equation}\label{ffs32d0d1k-3}
    u(a,\ot)\geq u(x,\ot)-\vep
\end{equation}
whenever $|x-a|<\delta$. But using the upper bound on $u_t$, we have
\begin{equation}\label{ffs32d0d1k-4}
    u(x,\ot)\geq  u(x,t_0-)-C(t_0-\ot) \geq  u(x,t_0-) -C\vep.
\end{equation}
Hence, by \eqref{ffs32d0d1k-1} and the fact that $x-a<\delta<\vep$, we have
\begin{equation}
    u(a,t_0-) \leq u(x,t_0-)+C(x-a)\leq u(x,t_0-)+C\vep,
\end{equation}
while \eqref{ffs32d0d1k-2}, \eqref{ffs32d0d1k-3}, and \eqref{ffs32d0d1k-4} give
\begin{equation}
    u(a,t_0-) \geq u(x,t_0-)-C\vep, \quad x\in(a,a+\delta).
\end{equation}
\end{proof}

\section{Nondegeneracy estimates for physical solutions and applications} \label{sec: estimates physical}
In the present section, we will exploit the estimates of Section \ref{sec:estimates}, combined with the physical jump condition \eqref{eq: physical defi}, to obtain nondegeneracy results for physical solutions. As applications, we prove that the freezing time $s$ is locally Lipschitz continuous, as well as Theorem \ref{thm: local to global uniqueness}. 
We begin with an elementary consequence of \eqref{eq: physical defi}.
\begin{lem} \label{lem: u large jump} Let $\Lambda$ be a physical solution to \eqref{supercooled eqs}, and let $u$ be as in \eqref{u defi}. Let $t_0>0$, and assume that $\Lambda_{t_0}>\Lambda_{t_0-}$. Then $u(\Lambda_{t_0-},t_0-)\geq \alpha^{-1}$ and $u(\Lambda_{t_0},t_0-)\in (0, \alpha^{-1}]$. Moreover, there exists $x\in (\Lambda_{t_0-},\Lambda_{t_0})$ such that $u_x(x,t_0-)<0$.  
\end{lem}
\begin{proof} These claims follow directly from \eqref{u(t-) formula}, Lemmas \ref{lem: smooth ext} and \ref{lem: lim past regu}, and the definition of physical solution. Indeed, for $x>0$,
\begin{equation}
    \P(\tau \geq t_0, B_{t_0}\in (\Lambda_{t_0-},\Lambda_{t_0-}+x])=\int_{\Lambda_{t_0-}}^{\Lambda_{t_0-}+x}u(y,t_0-)dy,
\end{equation}
so a value $u(\Lambda_{t_0-},t_0-)<\alpha^{-1}$ would contradict the physical jump condition \eqref{eq: physical defi}, which proves that 
\begin{equation} \label{ualrgewr123}u(\Lambda_{t_0-},t_0-)\geq\alpha^{-1}.\end{equation} On the other hand, we have
\begin{equation} 
    \int_{\Lambda_{t_0-}}^{\Lambda_{t_0}}u(x,t_0-)dx=\P(\tau= t_0)=\alpha^{-1}(\Lambda_{t_0}-\Lambda_{t_0-}).
\end{equation}
Thus, \eqref{eq: physical defi} implies that 
\begin{equation} \label{uasmllwr123}
   u(\Lambda_{t_0},t_0-)=\lim_{x \downarrow \Lambda_{t_0}} \frac{1}{x}\int_{\Lambda_{t_0}}^{\Lambda_{t_0}+x}u(y,t_0-)dy \leq \alpha^{-1} .
\end{equation}
Note also that $u(\Lambda_{t_0},t_0-)>0$, by Lemma \ref{lem: smooth ext} and the strong maximum principle.
Finally, if we had $u_x(x,t_0-) \geq 0$ for $x\in (\Lambda_{t_0-},\Lambda_{t_0})$, \eqref{ualrgewr123} and \eqref{uasmllwr123} would imply that $u_x(\cdot ,t_0-) \equiv 0$ on  $(\Lambda_{t_0-},\Lambda_{t_0})$. This is a contradiction to the fact that, by Lemma \ref{lem: smooth ext}, $u(\cdot,t_0-)$ is analytic.
\end{proof}
Next, we obtain a local lower bound for the density of stopped particles.
\begin{lem} \label{lem: nu lower bd} Let $\Lambda$ be a physical solution, let $u$ be defined by \eqref{u defi}, and let $\nu$ be defined by \eqref{nu defi}. The function $x \mapsto \nu(x)$ is locally bounded away from $0$ in $(\Lambda_0,\alpha)$.
\end{lem}
\begin{proof} Let $x_0\in (\Lambda_0,\alpha)$, and let $t_0 =s(x_0)>0$.
     By Lemma \ref{lem: uxx ux bd}, there exists $C_1>1$, such that 
\begin{equation} \label{wopdcqpo123}  u_x(x,t) \geq-C_1, \quad (x,t)\in ((x_0-10\alpha,x_0+10\alpha)\times (t_0/10,\infty))\cap \{u>0\}.\end{equation}
Assume, by contradiction, that there exists a sequence $x_n \to x_0$ such that $\nu(x_n)\to 0$ as $n\to \infty$. Recalling that
\begin{equation}
    J=\{x\in (\Lambda_0,\alpha): \Lambda_{t-}<x<\Lambda_{t} \text{ for some } t\in (0,\infty)\}, 
\end{equation}
we have, by definition, $\nu(x) =\alpha^{-1}$ for $x\in (\Lambda_0,\alpha)\backslash J$.  Therefore, we may assume that $x_n\in J$, namely that \[ \nu(x_n)=u(x_n,s(x_n)-), \quad \Lambda_{s(x_n)-}<\Lambda_{s(x_n)}, \quad x_n \in (\Lambda_{s(x_n)-},\Lambda_{s(x_n)}).\]
But, by Lemma \ref{lem: u large jump}, $u(\Lambda_{s(x_n)-},s(x_n)-)\geq \alpha^{-1}$, so that, by \eqref{wopdcqpo123}, for sufficiently large $n$,
\begin{equation} 
   \nu(x_n)=u(x_n,s(x_n)-)\geq u(\Lambda_{s(x_n)-},s(x_n)-)-C_1(x_n-\Lambda_{s(x_n)-})\geq \alpha^{-1} -C_1(x_n-\Lambda_{s(x_n)-}),  
\end{equation}
which, since $\nu(x_n)\to 0$, implies that
\[\Lambda_{s(x_n)}-\Lambda_{s(x_n)-}\geq x_n-\Lambda_{s(x_n)-}\geq \frac{1}{2\alpha C_1},\]
for $n$ sufficiently large. Thus, since $x_n \in (\Lambda_{s(x_n)-},\Lambda_{s(x_n)})$ and $x_n \to x_0$,  we infer that $\Lambda_{s(x_0)-}<\Lambda_{s(x_0)}$ and, for $n$ sufficiently large,
\[ s(x_n)=s(x_0), \quad x_n\in (\Lambda_{s(x_0)-},\Lambda_{s(x_0)}), \;\;\text{ and } \;\; \;\nu(x_n)=u(x_n,s(x_0)-).\] But, this implies that $\nu(x_n) \to 0=u(x_0,s(x_0)-)$, a contradiction to Lemmas \ref{lem: smooth ext}, \ref{lem: lim past regu} and \ref{lem: u large jump}.
\end{proof}
Next, we derive the main nondegeneracy estimate for physical solutions.
\begin{prop}[Linear nondegeneracy] \label{prop: u nondegeneracy} Let $\Lambda$ be a physical solution, let $u$ be defined by \eqref{u defi}. Assume that $(x_0,t_0) \in \R \times (0,\infty)$, and let $r>0$ be such that $r^2<t_0/100$. There exists a small constant $c>0$ such that
\begin{equation}
    u(x,t)\geq c(x-\Lambda_t)^+, \quad (x,t)\in Q_r(x_0,t_0).
\end{equation}
\end{prop}
\begin{proof} By a compactness argument, it is enough to show that the statement holds for sufficiently small $r$. If $(x_0,t_0)$ is not a free boundary point, the conclusion follows trivially by the fact that the set $\{u>0\}$ is open and $x-\Lambda_t\leq0= u(x,t)$ whenever $(x,t)\in \{u=0\}^{\circ}$. Thus, we may assume that $x_0\in (\Lambda_0,\alpha)$, $t_0=s(x_0)$, and $r<\min(\alpha-x_0,x_0-\Lambda_0)/10$. By Lemma \ref{lem: nu lower bd}, there exists $c_0>0$ such that
\begin{equation} \label{nuposciaom}
    \nu(x)\geq c_0, \quad x\in (x_0-4r,x_0+4r).
\end{equation}
We now fix $(x_1,t_1)\in Q_{r}(x_0,t_0)$, and define, for $c_1, c_2>0$ to be chosen later,
\begin{equation}\label{z-def-nondeg}
    z(x,t)=u(x,t)-c_1(x-\Lambda_t) +c_2\left(\frac{c_0}{2}((x-x_1)^2+t_1-t)-w(x,t)\right), \quad (x,t)\in D:=Q_r^-(x_1,t_1).
\end{equation}
We have, in view of \eqref{w equation}, \eqref{nuposciaom} and the monotonicity of $\Lambda$, in the distributional sense,
\begin{equation} \label{zcladorci} z_t-\frac12z_{xx}=c_1\dot \Lambda_t+c_2(\nu(x)-c_0)\geq 0, \quad (x,t)\in Q_{4r}(x_0,t_0)\cap\{u>0\}.\end{equation}
Recalling \eqref{nu defi}, Lemmas \ref{lem: smooth ext} and \ref{lem: lim past regu}, and the continuity of $w$, we see from \eqref{nuposciaom} that if $c_1$ is sufficiently small, depending on $c_0$ and $\alpha$ but not on $(x_1,t_1)$, then
\begin{equation}\label{partk1o1}
    \liminf_{(x,t)\in D,\; \text{dist}((x,t),\{u=0\})\to0}z(x,t) \geq 0.
\end{equation}
Let $0<\rho\ll r$. If $(x,t)\in \partial_pD+Q_\rho(0,0)$, then either $|x-x_1|\geq r-\rho$ or $t\leq t_1-r^2+\rho^2$. Hence, after requiring $\rho\leq r/4$,
\begin{equation}\label{partk1o2a}
(x-x_1)^2+t_1-t\geq \frac{r^2}{2}.
\end{equation}
If, in addition, $(x,t)\in \{u=0\}+Q_\rho(0,0)$, choose $(y,s)\in \{u=0\}$ with $|x-y|<\rho$ and $|t-s|<\rho^2$. Since $w\geq0$, $w(y,s)=0$, and $w_x(y,s)=0$, the local bounds on $w_t$ and $w_{xx}$ give
\begin{equation}\label{partk1o2b}
0\leq w(x,t)\leq |w(x,t)-w(x,s)|+|w(x,s)-w(y,s)|\leq C\rho^2.
\end{equation}
Finally, on $D+Q_\rho(0,0)$, using $|x_1-x_0|<r$ and $\Lambda\geq0$, we have
\begin{equation}\label{partk1o2c}
(x-\Lambda_t)^+\leq 4(r+|x_0|).
\end{equation}
For $(x,t)\in(\partial_pD+Q_\rho(0,0))\cap(\{u=0\}+Q_\rho(0,0))$, the definition \eqref{z-def-nondeg}, the bounds \eqref{partk1o2a}, \eqref{partk1o2b}, and \eqref{partk1o2c}, and $u\geq0$ give
\begin{equation}\label{partk1o2d}
z(x,t)\geq -4c_1(r+|x_0|)+c_2\left(\frac{c_0r^2}{4}-C\rho^2\right).
\end{equation}
Requiring $c_1<\frac{c_0r^2}{64(r+|x_0|)}c_2$, and then choosing $0<\rho \ll r$ small enough (depending on $r$, $c_0$, and local bounds on $w_t$ and $w_{xx}$, but not on $c_1$, $c_2$, or $(x_1,t_1)$) that $C\rho^2\leq c_0r^2/16$, \eqref{partk1o2d} gives
\begin{equation} \label{partk1o2}z\geq -c_2\frac{c_0r^2}{16}-c_2C\rho^2 +c_2\frac{c_0r^2}{4}\geq c_2\frac{c_0r^2}{8}, \;\; (x,t)\in\left(\partial_{p} D+Q_{\rho}(0,0))\cap(\{u=0\}+Q_{\rho}(0,0)\right).\end{equation}
On the other hand, if $(x,t)\in Q_{4r}(x_0,t_0)\backslash(\{u=0\}+Q_{\rho}(0,0))$, then $u(x,t) \geq \delta$ for some $\delta>0$ that depends on $\rho$ and $u$, but not on $(x_1,t_1)$. We therefore have
\begin{equation}\label{partk1o3}
    z(x,t)\geq \delta -C(c_1+c_2)\geq \delta/2, \quad (x,t)\in (D+Q_{\rho}(0,0))\backslash(\{u=0\}+Q_{\rho}(0,0)),
\end{equation}
for sufficiently small $c_1$ and $c_2$, depending on $\delta$, $r$, $\|w\|_{L^\infty(Q_{4r}(x_0,t_0))}$, and $|x_0|$, but not on $(x_1,t_1)$.

We now have all the estimates needed to conclude with the maximum principle, which we justify in detail below by mollification. Fix $\eta>0$, and set
\[
A:=(D\cap\{u=0\})\cup\partial_pD.
\]
By \eqref{partk1o1}, and by \eqref{partk1o2}--\eqref{partk1o3} near the parabolic boundary, there exists $0<\sigma<\rho/10$ such that
\begin{equation}\label{partk1o5}
z(x,t)\geq -\eta \quad\text{for }(x,t)\in D\cap(A+Q_{4\sigma}(0,0)).
\end{equation}

Set
\[
D_{\sigma}:=D\cap\{t<t_1\}\setminus\overline{A+Q_{2\sigma}(0,0)}.
\]
Let
\begin{equation}\label{eq:backward-time-mollifier}
\psi_\varepsilon(y,s)=\varepsilon^{-3}\phi(y/\varepsilon)\theta(s/\varepsilon^2),
\end{equation}
where $\phi\in C_c^\infty((-1,1))$ and $\theta\in C_c^\infty((0,1))$ are nonnegative, with $\int\phi=\int\theta=1$, and $\varepsilon<\sigma$, and set
\[
z_\varepsilon(x,t):=(z*\psi_\varepsilon)(x,t)=\iint z(x-y,t-s)\psi_\varepsilon(y,s)\,dy\,ds.
\]
Thus only earlier times are sampled. For points of $D_{\sigma}$, the convolution only uses values of $z$ in $D\cap\{u>0\}$, so, by \eqref{zcladorci},
\begin{equation}\label{partk1o6}
(z_\varepsilon)_t-\frac12(z_\varepsilon)_{xx}
=\left(z_t-\frac12z_{xx}\right)*\psi_\varepsilon\geq0
\quad\text{in } D_{\sigma}.
\end{equation}
Moreover, on the parabolic boundary of $D_{\sigma}$, the mollifier samples only points in $D\cap(A+Q_{3\sigma}(0,0))$. Hence \eqref{partk1o5} gives $z_\varepsilon\geq-\eta$ on that boundary. Since $z_\varepsilon$ is smooth, the classical parabolic minimum principle \cite[Cor. 2.5]{LiebermanBook} yields $z_\varepsilon\geq-\eta$ in $D_{\sigma}$. Letting $\varepsilon\downarrow0$, we obtain $z\geq-\eta$ almost everywhere in $D_{\sigma}$. Together with \eqref{partk1o5} on $D\cap(A+Q_{2\sigma}(0,0))$, this gives
\begin{equation}\label{partk1o6a}
z\geq-\eta \quad\text{almost everywhere in }D.
\end{equation}
Since $\eta>0$ is arbitrary,
\begin{equation}\label{partk1o7}
z\geq0 \quad\text{almost everywhere in }D.
\end{equation}
If $x_1\leq\Lambda_{t_1}$, the desired estimate is trivial. Otherwise $(x_1,t_1)\in\{u>0\}$. In $\{u>0\}$, the function $z$ is upper semicontinuous, since $u$ is smooth, $w$ is continuous, and $\Lambda$ is upper semicontinuous. Therefore \eqref{partk1o7} gives $z(x_1,t_1)\geq0$, which yields
\[u(x_1,t_1)\geq c_1(x_1-\Lambda_{t_1})+c_2w(x_1,t_1)\geq c_1(x_1-\Lambda_{t_1}).\]
Since the choice of $(x_1,t_1)\in Q_{r}(x_0,t_0)$ was arbitrary, this concludes the proof.
\end{proof}
As an application, we can already show that the free boundary is locally Lipschitz continuous.
\begin{cor} \label{cor:s-lipschitz} The freezing time $s:(\Lambda_0,\alpha)\to (0,\infty)$ is locally Lipschitz continuous.
\end{cor} 

\begin{proof}
Let $\delta>0$ and let $\Lambda_0+\delta< x_1<x_2<\alpha-\delta$.  Assume that $r:=|x_1-x_2|<\delta/4$ and $s(x_2)\geq s(x_1)>\delta/4.$ Then, by Proposition \ref{prop: u nondegeneracy}, there exists a constant $c>0$, depending on $\delta$ but not on $x_1$ and $x_2$, such that for $t\in (s(x_1),s(x_2))$,
\[u(x_2+r,t)\geq c(x_2+r-\Lambda_{t})^+\geq c(x_2+r-\Lambda_{s(x_2)-})^+\geq cr.\]
Thus, since $w(x_1,s(x_1))=w_x(x_1,s(x_1))=0$ and, by Proposition \ref{prop: w eq}, $|w_{xx}|$ is bounded,
\begin{multline*}
cr(s(x_2)-s(x_1))\leq \int_{s(x_1)}^{s(x_2)}u(x_2+r,s)ds=\int_{s(x_1)}^{s(x_2)}-w_t(x_2+r,s)ds\\
\leq w(x_2+r,s(x_1))\leq C(x_2+r-x_1)^2\leq Cr^2,
\end{multline*}
that is,
\[|s(x_2)-s(x_1)|\leq \frac{C}{c}|x_2-x_1|.\]
\end{proof}
As a second application of the nondegeneracy estimate, we now prove that, for positive times $t>0$, any physical solution falls under the well-posedness regime of \cite[Thm. 1.1]{delarue}. Indeed, the existence, uniqueness, and regularity results in that work were obtained under the extra assumption that the initial density $u_0$ changes monotonicity at most finitely often on compact sets. We show that any physical solution automatically satisfies this local monotonicity property at any time level $t>0$, even if the initial data does not.
\begin{prop}\label{prop: monotonicity changes}Let $\Lambda$ be a physical solution to \eqref{supercooled eqs}, and let $u$ be given by \eqref{u defi}. Then, for every $t>0$, the function $u(\cdot,t-)$ changes monotonicity at most finitely often on compact subsets of $[\Lambda_{t-},\infty)$. 
\end{prop}
\begin{proof}
By Proposition \ref{prop: u eq}, $u\in L^{2}_{\operatorname{loc}}((0,\infty);H^1_{\operatorname{loc}}(\R)).$ Therefore, in particular, for almost every $t_1>0$, $u(\cdot,t_1)$ is in $H^{1}_{\operatorname{loc}}(\R)$ and is therefore continuous. This implies, for almost every $t_1>0$,
\begin{equation}\label{eq:continuous-slice-boundary-trace}
    \lim_{x\downarrow\Lambda_{t_1}}u(x,t_1)=0.
\end{equation}
By Lemmas \ref{lem: lim past regu} and \ref{lem: u large jump}, it follows that $\Lambda_{t_1}=\Lambda_{t_1-}$ and $u(\cdot,t_1)=u(\cdot,t_1-)$ for any such $t_1$. Furthermore, in view of Lemma \ref{lem: uxx ux bd} and Proposition \ref{prop: u nondegeneracy}, $\lim_{x \downarrow \Lambda_{t_1}}u_x(x,t_1)$ exists, and is strictly positive. On the other hand, since $u$ solves the heat equation in $\{u>0\}=\{(x,t):x>\Lambda_t\}$, $u(\cdot,t_1)$ is analytic in $(\Lambda_{t_1},\infty)$. From the strict monotonicity near the boundary and the interior analyticity, we infer that $u(\cdot,t_1)=u(\cdot,t_1-)$ changes monotonicity at most finitely often on compact subsets of $[\Lambda_{t_1},\infty)$. By \cite[Lem.\ 4.1]{delarue}, this monotonicity property propagates in time and is therefore valid for $u(\cdot,t-)$, for every $t\geq t_1$. Since $t_1$ may be chosen to be arbitrarily small, we conclude that the property is valid for all $t>0$.
\end{proof}
Finally, we may now prove that, for physical solutions, local uniqueness implies global uniqueness.
\begin{proof}[Proof of Theorem \ref{thm: local to global uniqueness}] Assume $\Lambda^1$ and $\Lambda^2$ are two physical solutions of \eqref{supercooled eqs}, and assume that $\Lambda^1_{t_0-}=\Lambda^2_{t_0-}$ for some $t_0>0$. Letting $u_1$ and $u_2$ be the corresponding functions given by \eqref{u defi}, we then have $u_1(\cdot,t_0-)=u_2(\cdot,t_0-)$. In addition, by Proposition \ref{prop: monotonicity changes}, $u_1(\cdot,t_0-)=u_2(\cdot,t_0-)$ changes monotonicity at most finitely many times on compact subsets of $[\Lambda^1_{t_0-},\infty)=[\Lambda^2_{t_0-},\infty)$. Therefore, by \cite[Thm. 1.1]{delarue}, $\Lambda^1_t=\Lambda^2_t$ for $t\geq t_0$.    
\end{proof}
\section{Jump times for physical solutions and backward propagation of oscillation} \label{sec: jumps}
This section is dedicated to the proof of Theorem \ref{thm: jumps}. Namely, we will prove that, for any physical solution $\Lambda$, no positive time $t_{0}$ can be an accumulation point for the set of discontinuities of $\Lambda$. In fact, for later use in the singular point analysis of Section \ref{sec:regularity}, we will also allow the jump length in \eqref{eq: physical defi} to degenerate to zero, and show the following stronger version of Theorem \ref{thm: jumps}.
\begin{thm}\label{thm: stronger jumps}
Let $\Lambda$ be a physical solution to \eqref{supercooled eqs}. Then the set of singular times is locally finite in $(0,\infty)$.
\end{thm}

We remark that, from Definition \ref{def: reg sing}, $t>0$ is a singular time if either $t$ is a jump time, or
\begin{equation}\label{eq: generalized jump}
\Lambda_t=\Lambda_{t-}\quad \text{and}\quad u(\Lambda_t,t-)=\alpha^{-1}.
\end{equation}
The second alternative is the zero-length analogue of a jump: $\Lambda$ remains continuous but the boundary trace reaches the critical value.

We first observe that, in view of the nondegeneracy estimates of the previous section, singular times cannot accumulate to the right of $t_0$.
\begin{lem} \label{lem: right accum} Let $\Lambda$ be a physical solution to \eqref{supercooled eqs}, let $u$ be given by \eqref{u defi}, and let $t_0>0$. There exists $\vep>0$ such that $\Lambda\in C^{1}((t_0,t_0+\vep))$, $u \in C(\R \times (t_0,t_0+\vep))$, and
\begin{equation} \label{velocity boundary}
    \dot \Lambda_t=\frac{\alpha}{2}\lim_{x \downarrow \Lambda_t}u_x(x,t)=:\frac{\alpha}{2}u_x(\Lambda_t,t)>0, \quad t\in (t_0,t_0+\vep).
\end{equation}
\end{lem}
\begin{proof}
 By Proposition \ref{prop: monotonicity changes}, for every $t>0$, $u(\cdot,t-)$ changes monotonicity at most finitely many times on compact subsets of $(\Lambda_t,\infty)$. The result then follows directly by applying \cite[Thm. 1.1]{delarue} to the unique physical solution on $[t_0,\infty)$ with initial density $u(\cdot,t_0-)/\int_{\R}u(\cdot,t_0-)$.
The continuity of $\Lambda$ excludes jumps on $(t_0,t_0+\vep)$, while the continuity of $u$ gives $u(\Lambda_t,t-)=u(\Lambda_t,t)=0$ there, excluding \eqref{eq: generalized jump}.
\end{proof}
\begin{rem}For the benefit of the reader, we observe that, formally, \eqref{velocity boundary} follows by differentiating with respect to time the identity
\begin{equation} \label{lambda u eq}\Lambda_{t}=\alpha\P(\tau\leq t) = \alpha\left(1-\int_{\Lambda_t}^{\infty}u(x,t)dx\right),\end{equation}
while using the fact that $u(\Lambda_t,t)=0$ and $u$ solves the heat equation in $\{x>\Lambda_t\}$.    
\end{rem}
The bulk of this section will then be spent showing that singular times cannot accumulate to the left of any $t_1>0$. The strategy will be to relate the occurrence of singular times with oscillatory behavior of $u$ in past times. For this purpose, we begin by citing the following result from \cite[Lem. 4.2]{delarue}, which gives a local description of the zero set of $u_x$ in the set $\{u>0\}$, and is valid for any local solution of the heat equation.
\begin{lem}[Local description of the set of critical points for the heat equation] \label{lem: zero set} Let $\Lambda$ be a physical solution to \eqref{supercooled eqs}, and let $u$ be given by \eqref{u defi}, and let $(x,t)\in \{u>0\}$. Let $k\in \mathbb{N}\backslash\{0\}$ be the smallest integer such that $\partial^{k+1}_xu(x,t) \neq 0$. Then there exists a neighborhood $Q_{\delta}(x,t)$ in which the zero set of $u_x$ is the union of $k$ curves. $2\lfloor k/2\rfloor$ curves are given by the graphs (in the $(t,x)$--coordinate system) of continuous functions on $[t-\delta,t]$ that are $C^{\infty}$ on $[t-\delta,t)$. In the $(x,t)$--coordinate system, these curves form $\lfloor k/2 \rfloor$ strictly convex smooth functions with graphs contained in the negative half-space. If $k$ is odd, there is one more curve given, in the $(t,x)$ coordinates, by the graph of a $C^{\infty}$ function on $[t-\delta,t+\delta]$ that, in particular, crosses the $x$-axis.
    
\end{lem}
We now show that, for any $t_0\in (0,\infty)$, any $t_1>t_0$ sufficiently close to $t_0$, and any point $(x_1,t_1)$ such that $u(x_1,t_1)>0$ and $u_x(x_1,t_1)=0$, the point $(x_1,t_1)$ can be connected back to $\overline{\{u(\cdot,t_0)>0\}}$ with a continuous curve that remains contained in $\{u>0\}\cap\{u_x=0\}$ for $t\in (t_0,t_1]$. In addition, this curve may be chosen so that $u_x \neq 0$ in either a left-neighborhood or a right-neighborhood of the curve. 
\begin{lem}[Local construction of zero curves] \label{lem:zerocurves} Let $\Lambda$ be a physical solution to \eqref{supercooled eqs}, and let $u$ be given by \eqref{u defi}. Let $t_0\in (0,\infty)$, and let $\vep>0$ be such that $\Lambda \in C^1((t_0,t_0+\vep))$, and $u_x(\Lambda_t,t)>0$ for $t\in (t_0,t_0+\vep)$. Let $t_1\in(t_0,t_0+\vep)$, and let $x_1\in (\Lambda_{t_1},\infty)$ be such that $u_x(x_1,t_1)=0$. There exist continuous curves $z^-,z^+ : [t_0,t_1]\to (0,\infty)$, such that
\begin{enumerate}
    \item[(a)] $z^{\pm}(t_1)=x_1$, and $z^-(t) \leq z^+(t)$ for all $t\in [t_0,t_1]$.
    \item[(b)] $z^{\pm}(t) \in (\Lambda_t,\infty)$ and $u_x(z^{\pm}(t),t)=0$ whenever $t\in (t_0,t_1]$.
\item[(c)] For every $\bar t\in (t_0,t_1]$, there exists $\delta>0$ such that
\(Q_{\delta}^-(z^\pm(\bar t),\bar t)\subset \{(x,t)\in (0,\infty)\times(t_0,t_1]:\Lambda_t < x\}\), and
\[
u_x\neq 0 \quad\text{on }\quad Q_{\delta}^-(z^-(\bar t),\bar t)\cap\{x<z^-(t)\}
\quad\text{and on } \quad Q_{\delta}^-(z^+(\bar t),\bar t)\cap\{x>z^+(t)\}.
\]
\end{enumerate} 
The curves $z^{\pm}=z^{\pm}_{(t_1,x_1)}$ are uniquely determined by $t_1$ and $x_1$. Furthermore, two curves $z^-_{(t_1,x_1)}$ and $z^-_{(t_2,x_2)}$ (or two curves $z^+_{(t_1,x_1)}$ and $z^+_{(t_2,x_2)}$)  do not intersect at any $t\in [t_0,\min(t_1,t_2)]$ unless they agree in all of $[t_0,\min(t_1,t_2)]$. 

\end{lem}
\begin{proof} This result is based on part of the proof of \cite[Lem.\ 4.1]{delarue}, where a truncated version of $z^-$ was constructed. Our proof is much shorter, since the estimates of Lemma \ref{lem: uxx ux bd} completely circumvent the lengthy technical discussion of \cite[Lem.\ 4.3]{delarue}. We give the full details of the construction, because they are needed to deduce additional new properties, such as the fact that the curves can be defined up to $t=t_0$, and with no need for spatial truncation. By Lemma \ref{lem: zero set}, there exists a parabolic neighborhood $Q_{\sqrt{\delta}}^-(x_1,t_1)$ such that the zero set of $u_x$ is given by the union of the graphs of finitely many smooth curves, parametrized by time. We can then define $z^-$ and $z^+$ on $[t_1-\delta,t_1]$ as the minimum and maximum, respectively, of these curves. This argument can be continued until some minimal time $\underline{t}\in [t_0,t_1]$, namely $\underline{t}=\inf \mathcal{T}_0$, where
\[\mathcal{T}_0:=\{r\in (t_0,t_1]: \text{(a), (b), and (c) hold, with} \;t_0 \;\text{replaced by } r, \text{ for some continuous}\; z^{\pm}:[r,t_1]\to (0,\infty)\}.\]
We observe now that, given $r\in \mathcal{T}_0$, both $z^-$ and $z^+$ are uniquely determined up to $t=r$, because otherwise (c) would be violated at the first time where two different candidates intersect. Therefore, the unique curves $z^-$ and $z^+$ can both be extended to $(\ut,t_1]$. By the intermediate value theorem, the set of limit points as $t\downarrow \underline{t}$ of either $z^-(t)$ or $z^+(t)$ forms a closed (possibly infinite) interval. Since $u(\cdot,\underline{t})$ is analytic and non-constant on $(\Lambda_{\ut},\infty)$, and $u_x$ vanishes along $z^{\pm}$, this closed interval must be a point, which means that $z^{\pm}$ can be extended continuously up to $t=\underline{t}$, although without yet ruling out the possibility that $z^{\pm}(\underline{t})=+\infty$.

We claim that two different curves of the same type (either of type $z^-$ or of type $z^+$) cannot meet on their common interval of definition, even if they have different terminal data. To keep track of domains, let $z_i$, $i=1,2$, be two such curves, with lower endpoints $\underline{t}_i$ and terminal times $t_i$, understood with the continuous extensions to $t=\underline{t}_i$ constructed above, and set
\[
t_*:=\max(\underline{t}_1,\underline{t}_2), \qquad t^*:=\min(t_1,t_2).
\]
If $t^*\leq t_*$, the conclusion is immediate. Thus we may assume $t_*<t^*$, and by the uniqueness for a fixed terminal point shown above we may also assume that $z_1(t^*)\neq z_2(t^*)$.
Assume, by contradiction, that two distinct curves $z_1$ and $z_2$ of the same type intersect in $[t_*,t^*)$, and let $s\in [t_*,t^*)$ be the largest intersection time, so that $y:=z_1(s)=z_2(s)$. With no loss of generality, we may assume that $z_1< z_2$ on $(s,t^*]$. If $y>\Lambda_s$, then $u(y,s)>0$, and we obtain a contradiction by applying the maximum principle to the caloric function $u_x$ in the region $\Omega:=\{(x,t):t\in (s,t^*), z_1(t)\leq x \leq z_2(t)\}$ bounded by the two curves, noting that $u_x\equiv 0$ on $\partial_{p}\Omega$. If $y=\Lambda_s$, the only issue is that $u_x$ may fail to be continuous at $(y,s)$. However, we can still deduce a contradiction from the maximum principle as long as we prove that
\begin{equation}
    \lim_{(x,t)\to (y,s), \;(x,t)\in \Omega} u_x(x,t)=0.
\end{equation}
But this follows from the spatial semiconcavity estimate of Lemma \ref{lem: uxx ux bd}, because $u_x(z_1(t),t)=u_x(z_2(t),t)=0$ and $\lim_{t \downarrow s} (z_2(t)-z_1(t))=0$. This completes the proof that distinct curves of the same type cannot intersect on their common interval $[t_*,t^*]$, unless they agree on this interval.

We next check that $z^{\pm}(\underline{t})<\infty$. It is enough to prove this for $z^+$, since $z^-\leq z^+$. This follows from a similar contradiction argument as above. To justify the maximum-principle step on an unbounded domain, we first note the following elementary decay estimate: for every $0<a<b<\infty$,
\begin{equation} \label{eq: ux decay infinity}
    \lim_{R\to\infty}\sup_{\substack{t\in [a,b]\\ x\geq R}} |u_x(x,t)|=0.
\end{equation}
To see this, we recall \eqref{uineqdokap} and split the integral at $M>0$,
\begin{multline*}
\sup_{\substack{t\in [a,b]\\ x\geq R}}u(x,t)
\leq \sup_{\substack{t\in [a,b]\\ x\geq R}}\int_0^M G(x-y,t)u_0(y)\,dy
+\sup_{\substack{t\in [a,b]\\ x\geq R}}\int_M^\infty G(x-y,t)u_0(y)\,dy\\
\leq \sup_{\substack{t\in [a,b]\\ x\geq R}}\int_0^M G(x-y,t)u_0(y)\,dy
+(2\pi a)^{-1/2}\int_M^\infty u_0(y)\,dy.
\end{multline*}
For fixed $M$, the first term tends to zero as $R\to\infty$, and the second term tends to zero as $M\to\infty$. Thus $u(x,t)\to0$ uniformly as $x\to\infty$ for $t\in [a,b]$. Since the frontier $\Lambda_t$ is bounded, and $u_x$ is caloric in $\{u>0\}$, the standard interior gradient estimate for the heat equation therefore gives \eqref{eq: ux decay infinity}.

Assume, by contradiction, that $z^+(\underline{t})=+\infty$. For $\rho>0$ and $R>\sup_{[\underline{t}+\rho,t_1]}z^+$, we consider the truncated domain
\[
\Omega_{\rho,R}:=\{(x,t):\underline{t}+\rho<t<t_1,\ z^+(t)<x<R\}.
\]
The function $u_x$ is caloric in $\Omega_{\rho,R}$ and vanishes on the lateral boundary $x=z^+(t)$, while \eqref{eq: ux decay infinity} gives $u_x\to0$ uniformly on the boundary $x=R$ as $R\to\infty$, and uniformly on the bottom boundary as $\rho\downarrow0$, since $z^+(\underline{t}+\rho)\to\infty$. Applying the maximum principle to both $u_x$ and $-u_x$ on $\Omega_{\rho,R}$, and then letting $R\to\infty$ and $\rho\downarrow0$, we obtain $u_x\equiv0$ in $\{(x,t):\underline{t}<t<t_1,\ z^+(t)<x\}$, a contradiction. Therefore $z^{-}(\underline{t})\leq z^+(\underline{t})<\infty$, as wanted.

Finally, we claim that the curves $z^{\pm}$ may be defined up to $t=t_0$, namely that $\underline{t}=t_0$. Since we have ruled out the possibility that $z^{\pm}(\underline{t})=\infty$, Lemma \ref{lem: zero set} implies that if $\underline{t}>t_0$, then $z^-(\underline{t})=\Lambda_{\underline{t}}$. To see that this is impossible, recall that, by the choice of $\vep$, $u_x(\Lambda_t,t)=\frac{2}{\alpha}\dot\Lambda_t>0$ for $t\in (t_0,t_1)$, and $\dot \Lambda$ is continuous on $(t_0,t_1)$. Thus, by Lemma \ref{lem: uxx ux bd} and Dini's theorem, $u_x>0$ in some small neighborhood $\{(x,t):t\in [\underline{t},t_1], x\in (\Lambda_{t},\Lambda_t+\delta)\}$, a contradiction to the fact that $u_x$ vanishes along $z^-$. This concludes the proof that $\underline{t}=t_0$.
\end{proof}
Next, we will exploit the local construction of zero curves to show backward propagation of oscillation. Namely, given $t_0<t_1$, and points $x_1<x_2$ with $u_x(x_1,t_1)<0<u_x(x_2,t_1)$, we show that one may propagate this monotonicity change back to time $t_0$ via curves along which $u_x$ has constant sign. As a first step, we prove that this is possible for short times.
\begin{figure}

\tikzset{every picture/.style={line width=0.75pt}} 

\begin{tikzpicture}[x=0.75pt,y=0.75pt,yscale=-1,xscale=1]

\draw  (12,266.81) -- (494.5,266.81)(60.25,23.81) -- (60.25,293.81) (487.5,261.81) -- (494.5,266.81) -- (487.5,271.81) (55.25,30.81) -- (60.25,23.81) -- (65.25,30.81)  ;
\draw    (59,49.31) -- (478.5,49.75) ;
\draw [color={rgb, 255:red, 208; green, 2; blue, 27 }  ,draw opacity=1 ]   (240.66,49.72) -- (386.83,150.48) ;
\draw    (60.25,266.81) .. controls (91,238.81) and (124.04,71.54) .. (167.04,49.54) ;
\draw    (221.83,49.15) .. controls (401.05,136.64) and (143.25,248.81) .. (60.25,266.81) ;
\draw    (96.83,266.15) .. controls (154.14,257.8) and (384.33,191.48) .. (404.83,161.15) ;
\draw [color={rgb, 255:red, 208; green, 2; blue, 27 }  ,draw opacity=1 ]   (388.83,162.48) -- (386.83,150.48) ;
\draw [color={rgb, 255:red, 208; green, 2; blue, 27 }  ,draw opacity=1 ]   (78.74,266.06) .. controls (99.83,265.81) and (357.83,189.48) .. (388.83,162.48) ;
\draw [color={rgb, 255:red, 74; green, 144; blue, 226 }  ,draw opacity=1 ]   (404.83,161.15) .. controls (410.76,146.8) and (377.83,123.15) .. (386.83,114.15) ;
\draw  [dash pattern={on 0.84pt off 2.51pt}]  (60.15,160.76) -- (483.56,162.53) ;

\draw (110,60) node [anchor=north west][inner sep=0.75pt]  [font=\scriptsize] [align=left] {$\displaystyle x=\Lambda _{t}$};
\draw (225,82) node [anchor=north west][inner sep=0.75pt]  [font=\scriptsize] [align=left] {$\displaystyle x=\tilde{z}^{-}( t)$};
\draw (488,274) node [anchor=north west][inner sep=0.75pt]  [font=\small] [align=left] {$\displaystyle x$};
\draw (66,11) node [anchor=north west][inner sep=0.75pt]  [font=\small] [align=left] {$\displaystyle t$};
\draw (269,179) node [anchor=north west][inner sep=0.75pt]  [font=\scriptsize,color={rgb, 255:red, 208; green, 2; blue, 27 }  ,opacity=1 ] [align=left] {$\displaystyle x=\gamma _{1}( t)$};
\draw (429,136) node [anchor=north west][inner sep=0.75pt]  [font=\scriptsize,color={rgb, 255:red, 74; green, 144; blue, 226 }  ,opacity=1 ] [align=left] {$\displaystyle x=f( t)$};
\draw (368,193) node [anchor=north west][inner sep=0.75pt]  [font=\scriptsize] [align=left] {$\displaystyle x=z^{-}( t)$};
\draw (26,46) node [anchor=north west][inner sep=0.75pt]  [font=\scriptsize] [align=left] {$\displaystyle t=t_{1}$};
\draw (25,250) node [anchor=north west][inner sep=0.75pt]  [font=\scriptsize] [align=left] {$\displaystyle t=t_{0}$};
\draw (29,154) node [anchor=north west][inner sep=0.75pt]  [font=\scriptsize] [align=left] {$\displaystyle t=\underline{t}$};

\end{tikzpicture}

\caption{Short time construction of $\gamma_1(t)$. The zero curve $\tilde{z}^-$ acts as a barrier that precludes $z^-$ from hitting the boundary at $t=t_0$.}
\label{fig:zerocurves}
\end{figure}
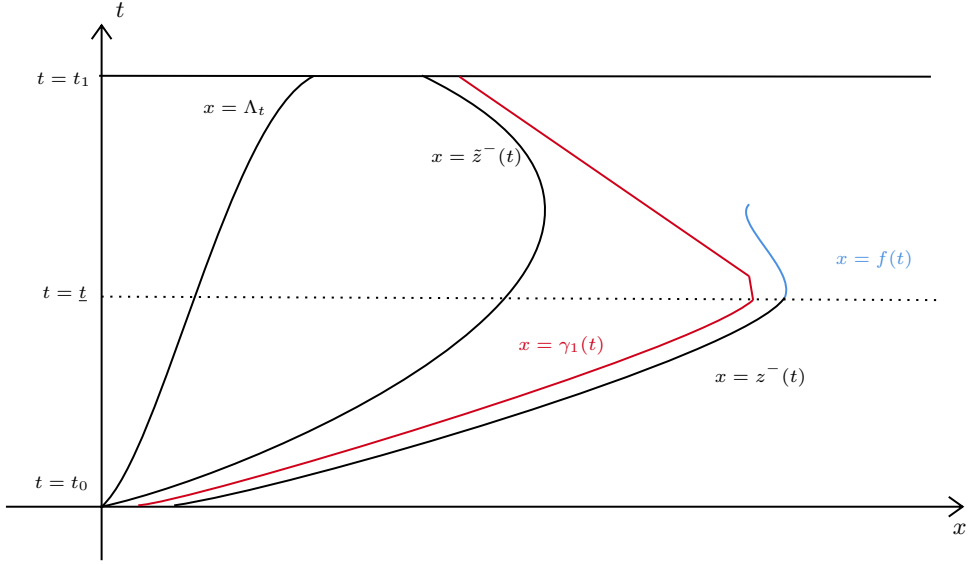
\begin{lem}[Local construction of curves of constant monotonicity] \label{lem:local curves} Let $\Lambda$ be a physical solution to \eqref{supercooled eqs}, let $u$ be given by \eqref{u defi}, and let $t_0\in (0,\infty)$. There exists $\vep>0$ such that the following holds for every $t_1\in (t_0,t_0+\vep)$. Assume that $x_1\in (\Lambda_{t_1},\infty)$ satisfies $u_x(x_1,t_1)<0$. Then there exists a continuous function $\gamma_1:[t_0,t_1]\to [0,\infty)$ such that $\Lambda_t<\gamma_1(t)$, $\gamma_1(t_1)=x_1$, and $u_x(\gamma_1(t),t)<0$ for all $t \in [t_0,t_1]$. Furthermore, if $x_2>x_1$ is such that $u_x(x_2,t_1)>0$, there exists a continuous function $\gamma_2:[t_0,t_1] \to [0,\infty)$ such that $\Lambda_t <\gamma_1(t) < \gamma_2(t)$, $\gamma_2(t_1)=x_2,$ and $u_x(\gamma_2(t),t)>0$ for all $t\in [t_0,t_1]$.
\end{lem}
\begin{proof} Let $\vep>0$ be as in Lemma \ref{lem: right accum}. We first construct the curve $\gamma_1$. 
Using the fact that $u_x(\Lambda_{t_1},t_1)>0$ and $u_x(x_1,t_1)<0$, there must exist $x_0\in (\Lambda_{t_1},x_1)$ such that $u_x(x_0,t_1)=0$. Let $\tilde{z}^-:[t_0,t_1]\to (0,\infty)$ be the corresponding zero curve satisfying $\tilde{z}^-(t_1)=x_0$ given by Lemma \ref{lem:zerocurves}.  Let $M>1$ be sufficiently large that the curves $t\mapsto (\tilde{z}^-(t),t)$ and $t \mapsto (x_1+M(t_1-t),t)$ do not intersect for $t\in [t_0,t_1]$, and let
\begin{equation} \ut=\inf\{t\in [t_0,t_1]: u_x(x_1+M(t_1-s),s)<0,\,\, \text{ for all } s\in[t,t_1]\}. \end{equation}
Since $u$ is smooth away from the boundary, $t_0\leq \ut <t_1$. If we had $u_x(x_1+M(t_1-\underline{t}),\ut)<0$, then $\ut=t_0$ and we may simply take $\gamma_1(t)\equiv x_1+M(t_1-t)$. We may assume then that $u_x(x_1+M(t_1-\underline{t}),\ut)=0$. 

We define the curve $\gamma_1$ in two steps: first on the interval $[\ut,t_1]$, and then on the interval $[t_0,\ut]$. In view of Lemma \ref{lem: zero set}, there exists $\delta>0$ such that the set \[\{u_x=0\} \cap \left((x_1+M(t_1-\underline{t})-\delta,x_1+M(t_1-\underline{t})+\delta)\times [\ut,\ut+\delta)\right)\] equals either the singleton $\{(x_1+M(t_1-\underline{t}),\ut)\}$ or $\{(f(t),t):t\in [\ut,\ut+\delta)\}$ for some smooth function $f:[\ut,\ut+\delta)\to \R$ such that $f(\ut)=x_1+M(t_1-\underline{t})$. In either case, we define the curve $\gamma_1$ on $[\ut,t_1]$ by letting $\gamma_1(t)=x_1+M(t_1-t)$ for $t \in [\ut+\delta/(2M),t_1]$, and letting $\gamma_1$ equal a linear function on $ [\ut, \ut+\delta/(2M)]$, chosen in such a way that $\gamma_1(t)\in(x_1+M(t_1-\underline{t})-\delta/2,x_1+M(t_1-\underline{t})+\delta/2)$ and, if applicable, $\gamma_1(t)\neq f(t)$ for all $t\in [\ut,\ut+\delta)$ (see Figure \ref{fig:zerocurves}).

If $\ut=t_0$, then the construction is complete, so we assume henceforth that $\ut>t_0$. Let $z^{\pm}:[t_0,\ut]\mapsto [0,\infty)$ be the two zero curves obtained by taking  $t_1:=\ut$ in Lemma \ref{lem:zerocurves}. We claim first that $z^-(t_0)>\Lambda_{t_0}$. Indeed, by our choice of $M$, the curves $z^-$ and $\tilde{z}^-$ are distinct, and therefore, by Lemma \ref{lem:zerocurves}, they do not intersect, which implies that $\Lambda_{t_0}\leq \tilde{z}^-(t_0)<z^-(t_0)$.

Now, depending on whether $\gamma_1(\ut)<x_1+M(t_1-\underline{t})$ or $\gamma_1(\ut)>x_1+M(t_1-\underline{t})$ we construct $\gamma_1$ on $[t_0,\ut]$ so that $(\gamma_1(t),t)$ remains, respectively, either on the left side of a sufficiently narrow open neighborhood of the graph of $z^-$ or on the right side of a sufficiently narrow open neighborhood of the graph of $z^+$ (Figure \ref{fig:zerocurves} depicts the case $\gamma_1(\underline{t})<x_1+M(t_1-\underline{t})$). Such neighborhoods may be chosen in view of the fact that $z^+(t_0)\geq z^-(t_0)>\Lambda_{t_0}$, ensuring that the curves $z^{\pm}$ remain bounded away from the free boundary. The fact that $u_x(\gamma_1(t),t)<0$ then follows from property (c) in Lemma \ref{lem:zerocurves}, combined with an additional application of Lemma \ref{lem: zero set} to the point $(z^-(t_0),t_0)$, provided that the chosen neighborhood of the  graph is sufficiently narrow.

Finally, if $x_2>x_1$ satisfies $u_x(x_2,t_1)>0$, then the construction of the curve $\gamma_2$ proceeds in the same way, except without the need to appeal to the auxiliary barrier $(\tilde{z}^-(t),t)$, since the curve $(\gamma_1(t),t)$ itself, along which $u_x<0$, keeps the entire construction away from the boundary.
\end{proof}
The next result shows that backward propagation of oscillation holds for arbitrarily long time intervals.
\begin{lem}[Global construction of curves of constant monotonicity] \label{lem:global curves} Let $\Lambda$ be a physical solution to \eqref{supercooled eqs}, let $u$ be given by \eqref{u defi}, and let $t_0\in (0,\infty)$. The conclusion of Lemma \ref{lem:local curves} holds for every choice of $t_1 \in (t_0,\infty)$.
\end{lem}
\begin{proof} For each $t\in (t_0,\infty)$, we denote by $P(t)$ the statement that the conclusion of Lemma \ref{lem:local curves} holds for $t_1:=t$. Assume, by contradiction, that $P(t)$ is false for some $t>t_0$. Then
\be t^*=\sup\{t\in [t_0,\infty): P(s)\text{ holds for all } s\in(t_0,t) \}<\infty. \ee 
By Lemma \ref{lem:local curves}, $t^*>t_0$. We claim that $P(t^*)$ holds. To show this, let $x_1 \in (\Lambda_{t^*},\infty)$ be such that $u_x(x_1,t^*)<0$. Then, since $u$ is smooth in the interior of its support, there exists $\delta>0$ such that $u_x<0$ on $Q_{\sqrt\delta}(x_1,t^*)$. Thus, $u_x<0$ along the straight line segment $L_1$ connecting $(x_1,t^*)$ and $(x_1,t^*-\delta/2)$. But, by the definition of $t^*$, it must be the case that $P(t^*-\delta/2)$ holds, and thus there exists some point $(t_0,y)$ and a continuous curve $t \mapsto (\tilde{\gamma}_1(t),t)$ connecting $(t^*-\delta/2,x_1)$ and $(t_0,y)$, along which $u_x<0$. Concatenating $\tilde{\gamma}_1$ with $L_1$, we obtain $\gamma_1:[t_0,t^*] \to (0,\infty)$ with the desired properties. If applicable, we can then construct $\gamma_2:[t_0,t^*]\to (0,\infty)$ in the same way. Therefore, we conclude that $P(t^*)$ holds. Now, again by Lemma \ref{lem:local curves}, followed by concatenation and the fact that $P(t^*)$ holds, we infer that $P(t^*+\vep)$ holds for all sufficiently small $\vep>0$. But this contradicts the maximality of $t^*$.
\end{proof}
The following lemma relates singular times to past oscillations. Given a curve $\gamma$ satisfying \eqref{gam4210d}, each singular time $t_*\in (t_0,t_{\infty})$ that is followed by another singular time in $(t_*,t_{\infty}]$ forces an oscillation of $u(\cdot,t_*-)$ to the left of $\gamma(t_*)$. The curves of Lemma \ref{lem:global curves} propagate this oscillation backwards to time $t_0$, at which Proposition \ref{prop: monotonicity changes} gives only finitely many oscillations in any compact interval.
\begin{lem}[Future singularity implies past oscillation] \label{lem: jump implies oscillation}  Let $\Lambda$ be a physical solution to \eqref{supercooled eqs}, let $u$ be given by \eqref{u defi}, let $0<t_0<t_{\infty}<\infty$, and assume that there exists a continuous curve $\gamma: [t_0,t_{\infty}]\to (0,\infty)$ such that 
\begin{equation} \label{gam4210d} u(\gamma(t),t)>0, \quad u_x(\gamma(t),t)<0, \quad t\in [t_0,t_{\infty}), \quad\Lambda_{t_{\infty}-}<\gamma(t_{\infty}).\end{equation}
Suppose also that $t_*\in (t_0,t_{\infty})$ is a singular time, and assume that $t_*$ is not the largest such time in $(t_0,t_{\infty}]$.
 Then there exist continuous curves $\gamma_1,\gamma_2:[t_0,t_*]\to (0,\infty)$ such that
\begin{equation} \label{curvcond1202}\gamma_1(t)<\gamma_2(t)<\gamma(t), \quad t\in [t_0,t_*], \quad \Lambda_{t_*-}<\gamma_1(t_*),\end{equation}
and
\begin{equation} \label{curvcond1203}u(\gamma_1(t),t),\;u(\gamma_2(t),t)>0, \quad u_x(\gamma_1(t),t)<0, \quad u_x(\gamma_2(t),t)>0, \quad t\in [t_0,t_*). \end{equation}
\end{lem}
\begin{proof}
If $\Lambda_{t_*}>\Lambda_{t_*-}$, then Lemma \ref{lem: u large jump} gives $u(\Lambda_{t_*},t_*-)\leq \alpha^{-1}$. If instead $\Lambda_{t_*}=\Lambda_{t_*-}$, then \eqref{eq: generalized jump} gives $u(\Lambda_{t_*},t_*-)=\alpha^{-1}$, and in particular $u(\Lambda_{t_*},t_*-)\leq \alpha^{-1}$.

We first choose the endpoint $x_2$ of the curve $\gamma_2$. We claim that there exists $x_2\in (\Lambda_{t_*},\gamma(t_*))$ such that $u_x(x_2,t_*-)=u_x(x_2,t_*)>0$. Indeed, assume by contradiction that $u(\cdot,t_*)$ is nonincreasing on $(\Lambda_{t_*},\gamma(t_*))$.
Together with $u(\Lambda_{t_*},t_*-)\leq \alpha^{-1}$, this gives
\begin{equation}\label{eq:bottom-bound-after-jump}
u(x,t_*)\leq \alpha^{-1},\quad \Lambda_{t_*}\leq x\leq \gamma(t_*).
\end{equation}
We next propagate \eqref{eq:bottom-bound-after-jump} forward in time, in strict form, by a maximum-principle argument. To avoid technicalities coming from low regularity, we proceed with a simple continuation argument that takes advantage of the short-time regularity given by Lemma \ref{lem: right accum}.
Let 
\begin{eqnarray}
    U(t):=\sup_{x\in [\Lambda_{t_*},\gamma(t)]}u(x,t)=\sup_{x\in [\Lambda_{t},\gamma(t)]}u(x,t), \qquad t\in [t_{*},t_{\infty}).
\end{eqnarray}
 We claim that $U$ is a strictly decreasing function. By Proposition \ref{prop: semiconvexity}, $u_t$ is locally bounded from above. Since $u$ is also smooth away from the free boundary, this implies that $U$ cannot jump upward, namely
\begin{equation}\label{eq:U-no-upward-jump}
U(t)\leq \liminf_{r\uparrow t}U(r), \qquad t\in (t_*,t_{\infty}).
\end{equation}
Thus, it is enough to show that $U$ is locally decreasing, in the sense that for every $t_1\in [t_*,t_{\infty})$, one has  $U(t_1+\delta)<U(t_1)$ for all $\delta>0$ sufficiently small.  Fix such a $t_1$. Since $u_t$ is locally bounded from above, we have
\begin{equation}\label{eq:limsup-corner-mp}
\limsup_{\substack{(x,t)\to (\Lambda_{t_1},t_1)\\ t>t_1,\; \Lambda_{t}<x<\gamma(t)}}u(x,t)\leq U(t_1).
\end{equation}

By Lemma \ref{lem: right accum}, we may choose $\delta>0$ so small that $t_1+2\delta<t_{\infty}$ and, on $(t_1,t_1+2\delta)$, the function $\Lambda$ is $C^1$ and $u$ is continuous. Let
\[
S_{\delta}:=\{(x,t):t_1<t\leq t_1+\delta,\; \Lambda_t<x<\gamma(t)\},
\]
and let
\[
\Gamma:=\partial_p S_{\delta}\setminus \{(\gamma(t),t):t_1<t<t_1+\delta\}
\]
denote the left and bottom sides of the parabolic boundary of $S_{\delta}$. For every $(x,t)\in \Gamma$, we have
\begin{equation}\label{eq:short-time-boundary-limsup}
\limsup_{\substack{(y,s)\to (x,t)\\ (y,s)\in S_{\delta}}}u(y,s)\leq U(t_1).
\end{equation}
Indeed, at points on the bottom side  this follows from the definition of $U$ and the interior continuity of $u$, at the bottom-left corner it follows from \eqref{eq:limsup-corner-mp}, and on  the moving left side the function $u$ vanishes continuously.

Together with \eqref{eq:short-time-boundary-limsup} and the condition $u_x<0$ on the right side $x=\gamma(t)$, the classical maximum principle applied to the caloric function $u$ gives $u\leq U(t_1)$ in $S_{\delta}$. Since $u$ is not constant, the strong maximum principle gives $u<U(t_1)$ in $S_{\delta}$. Since $u$ vanishes continuously on the left boundary and $u_x<0$ on the right boundary, we deduce by compactness that 
\begin{equation}\label{eq:strict-slice-after-jump}
U(t_1+\delta)=\sup_{x\in [\Lambda_{t_1+\delta},\gamma(t_1+\delta)]} u(x,t_1+\delta)< U(t_1)
\end{equation}
for all sufficiently small $\delta>0$, as wanted. Combined with the no-upward-jump property \eqref{eq:U-no-upward-jump}, this proves that $U$ is strictly decreasing in $[t_*,t_{\infty})$.

By the non-maximality of $t_*$, there exists a singular time $t_{**}\in (t_*,t_{\infty}]$. Note that, by \eqref{eq:bottom-bound-after-jump}, $U(t_*)\leq \alpha^{-1}$. Thus, by definition of $U$ and its strict monotonicity, we infer that 
\begin{equation}\label{eq:strict-drop-after-jump}
      u(x,t_{**}-)\leq \lim_{t\uparrow t_{**}} U(t)< U((t_{*}+t_{**})/2)<\alpha^{-1}, \qquad \Lambda_{t_*}<x<\gamma(t_{**}).
\end{equation}
 Applying \eqref{eq:strict-drop-after-jump} at $x=\Lambda_{t_{**}-}$ gives $u(\Lambda_{t_{**}-},t_{**}-)<\alpha^{-1}$. If $\Lambda_{t_{**}}>\Lambda_{t_{**}-}$,  this contradicts Lemma \ref{lem: u large jump}. If instead $\Lambda_{t_{**}}=\Lambda_{t_{**}-}$, it contradicts \eqref{eq: generalized jump}. This contradiction proves the existence of $x_2\in (\Lambda_{t_*},\gamma(t_*))$ such that $u_x(x_2,t_*-)=u_x(x_2,t_*)>0.$ 
Next, we choose the endpoint $x_1$ of the curve $\gamma_1$. If $\Lambda_{t_*}>\Lambda_{t_*-}$, then Lemma \ref{lem: u large jump} gives $x_1\in(\Lambda_{t_*-},\Lambda_{t_*})$ such that $u_x(x_1,t_*-)<0$. On the other hand, if $\Lambda_{t_*}=\Lambda_{t_*-}$, we choose $x_1$ as follows. By Proposition \ref{prop: monotonicity changes}, there exists $\rho>0$ such that $u(\cdot,t_*-)$ is monotone on $(\Lambda_{t_*},\Lambda_{t_*}+\rho)$. Shrinking $\rho$ if necessary, we may assume that $\Lambda_{t_*}+\rho<x_2$. It cannot be nondecreasing there, since then \eqref{eq: generalized jump} and the right-continuity of $u(\cdot,t_*-)$ at $\Lambda_{t_*}$ from Lemma \ref{lem: lim past regu} would give $\int_{\Lambda_{t_*}}^{\Lambda_{t_*}+x}u(y,t_*-)dy\geq\alpha^{-1}x$ for all sufficiently small $x>0$, contradicting \eqref{eq: physical defi} and $\Lambda_{t_*}=\Lambda_{t_*-}$. Hence there exists $x_1\in(\Lambda_{t_*},\Lambda_{t_*}+\rho)$ such that $u_x(x_1,t_*-)<0$.

We have then proved that there exist points $x_1, x_2$, with $\Lambda_{t_*-}<x_1<x_2<\gamma(t_*)$ such that $u_x(x_1,t_*-)<0$ and $u_x(x_2,t_*-)>0$. By Lemma \ref{lem: smooth ext}, we have for some small $\vep>0$,
\begin{equation} \label{0qkiw0qdedqequ}(-1)^{i+1}u_x(x,t)<0, \quad x_1<x_2<\gamma(t),\quad (x,t)\in Q^-_{\sqrt{\vep}}(x_i,t_*)\cap\{t<t_*\}, \quad i\in \{1,2\}.\end{equation}
Therefore, by Lemma \ref{lem:global curves}, there exist curves $\gamma_1,\gamma_2:[t_0,t_*-\vep/2]\to (0,\infty)$, $\gamma_1<\gamma_2$, such that $\gamma_1(t_*-\vep/2)=x_1$, $\gamma_2(t_*-\vep/2)=x_2$, and \eqref{curvcond1203} holds for $t\in [t_0,t_*-\vep/2]$. Since $u_x(\gamma_2(t),t)>0>u_x(\gamma(t),t)$ and $\gamma_2(t_*-\vep/2)=x_2<\gamma(t_*-\vep/2)$, we have $\gamma_1<\gamma_2<\gamma$. Finally, in view of \eqref{0qkiw0qdedqequ}, we may then extend $\gamma_1$ and $\gamma_2$ continuously to $[t_0,t_*]$, while still satisfying \eqref{curvcond1202} and \eqref{curvcond1203}, by simply letting them be constant on $[t_*-\vep/2,t_*]$.
\end{proof}
We can now prove the main result of this section by iterating Lemma \ref{lem: jump implies oscillation}.

\begin{figure}

\tikzset{every picture/.style={line width=0.75pt}} 

\begin{tikzpicture}[x=0.75pt,y=0.75pt,yscale=-1,xscale=1]

\draw  (50,254.5) -- (554.5,254.5)(100.45,16) -- (100.45,281) (547.5,249.5) -- (554.5,254.5) -- (547.5,259.5) (95.45,23) -- (100.45,16) -- (105.45,23)  ;
\draw    (99,62) -- (520.5,62.67) ;
\draw    (100.45,254.5) .. controls (140.45,224.5) and (128.17,189.33) .. (142.92,189.01) ;
\draw    (142.92,189.01) -- (172.06,188.82) ;
\draw    (172.06,188.82) .. controls (185.17,188.33) and (181.83,140) .. (201.29,140.97) ;
\draw    (201.29,140.97) -- (222.57,140.53) ;
\draw    (222.57,140.53) .. controls (233.83,140) and (225.83,113) .. (237.57,112.53) ;
\draw    (237.57,112.53) -- (250.57,112.53) ;
\draw    (250.57,112.53) .. controls (262.17,113.33) and (259.17,89.33) .. (269.21,90.31) ;
\draw    (386.5,63) .. controls (356.33,104.67) and (429.88,241.8) .. (387.83,254.67) ;
\draw [color={rgb, 255:red, 208; green, 2; blue, 27 }  ,draw opacity=1 ]   (347.83,254) .. controls (277.83,238) and (259.83,135) .. (260.17,112.33) ;
\draw [color={rgb, 255:red, 74; green, 144; blue, 226 }  ,draw opacity=1 ]   (322.5,254.33) .. controls (234.5,232.33) and (243.83,147) .. (243.83,112.67) ;
\draw [color={rgb, 255:red, 74; green, 144; blue, 226 }  ,draw opacity=1 ]   (198.5,254.67) .. controls (223.17,239.33) and (208.83,170) .. (211.93,140.75) ;
\draw [color={rgb, 255:red, 208; green, 2; blue, 27 }  ,draw opacity=1 ]   (277.17,254.33) .. controls (238.83,237) and (229.83,159) .. (230.17,140.33) ;
\draw [color={rgb, 255:red, 208; green, 2; blue, 27 }  ,draw opacity=1 ]   (174.5,254.33) .. controls (149.83,225) and (183.83,212) .. (185.17,189.33) ;
\draw [color={rgb, 255:red, 74; green, 144; blue, 226 }  ,draw opacity=1 ]   (112.83,254) .. controls (170.83,239) and (154.83,213) .. (157.49,188.91) ;
\draw  [dash pattern={on 0.84pt off 2.51pt}]  (101.57,112.03) -- (519.5,112.67) ;
\draw  [dash pattern={on 0.84pt off 2.51pt}]  (100.57,141.03) -- (518.5,141.67) ;
\draw  [dash pattern={on 0.84pt off 2.51pt}]  (100.57,189.03) -- (518.5,189.67) ;

\draw (280,69) node [anchor=north west][inner sep=0.75pt]  [font=\normalsize] [align=left] {$\displaystyle \vdots $};
\draw (135,141) node [anchor=north west][inner sep=0.75pt]  [font=\footnotesize] [align=left] {$\displaystyle x=\Lambda _{t}$};
\draw (379,105) node [anchor=north west][inner sep=0.75pt]   [align=left] {$ $};
\draw (409,118) node [anchor=north west][inner sep=0.75pt]  [font=\footnotesize] [align=left] {$\displaystyle x=\gamma ( t)$};
\draw (300,197) node [anchor=north west][inner sep=0.75pt]  [font=\tiny,color={rgb, 255:red, 208; green, 2; blue, 27 }  ,opacity=1 ] [align=left] {$\displaystyle x=\gamma _{2}^{3}( t)$};
\draw (212,95) node [anchor=north west][inner sep=0.75pt]  [font=\tiny,color={rgb, 255:red, 74; green, 144; blue, 226 }  ,opacity=1 ] [align=left] {$\displaystyle x=\gamma _{1}^{3}( t)$};
\draw (63,54) node [anchor=north west][inner sep=0.75pt]  [font=\scriptsize] [align=left] {$\displaystyle t=t_{\infty }$};
\draw (62,183) node [anchor=north west][inner sep=0.75pt]  [font=\scriptsize] [align=left] {$\displaystyle t=t_{1}$};
\draw (61,137) node [anchor=north west][inner sep=0.75pt]  [font=\scriptsize] [align=left] {$\displaystyle t=t_{2}$};
\draw (61,104) node [anchor=north west][inner sep=0.75pt]  [font=\scriptsize] [align=left] {$\displaystyle t=t_{3}$};
\draw (134,173) node [anchor=north west][inner sep=0.75pt]  [font=\tiny,color={rgb, 255:red, 74; green, 144; blue, 226 }  ,opacity=1 ] [align=left] {$\displaystyle x=\gamma _{1}^{1}( t)$};
\draw (177,126) node [anchor=north west][inner sep=0.75pt]  [font=\tiny,color={rgb, 255:red, 74; green, 144; blue, 226 }  ,opacity=1 ] [align=left] {$\displaystyle x=\gamma _{1}^{2}( t)$};
\draw (209.17,257.33) node [anchor=north west][inner sep=0.75pt]  [font=\tiny,color={rgb, 255:red, 208; green, 2; blue, 27 }  ,opacity=1 ] [align=left] {$\displaystyle x=\gamma _{2}^{2}( t)$};
\draw (157,257) node [anchor=north west][inner sep=0.75pt]  [font=\tiny,color={rgb, 255:red, 208; green, 2; blue, 27 }  ,opacity=1 ] [align=left] {$\displaystyle x=\gamma _{2}^{1}( t)$};
\draw (65,256) node [anchor=north west][inner sep=0.75pt]  [font=\scriptsize] [align=left] {$\displaystyle t=t_{0}$};
\draw (73,69) node [anchor=north west][inner sep=0.75pt]  [font=\normalsize] [align=left] {$\displaystyle \vdots $};
\draw (358,236) node [anchor=north west][inner sep=0.75pt]  [font=\normalsize] [align=left] {$\displaystyle \cdots $};
\draw (546,264) node [anchor=north west][inner sep=0.75pt]  [font=\footnotesize] [align=left] {$\displaystyle x$};
\draw (109,9) node [anchor=north west][inner sep=0.75pt]  [font=\footnotesize] [align=left] {$\displaystyle t$};

\end{tikzpicture}
\caption{Correspondence between singular times $t=t_1,t_2,\ldots,t_N$ and oscillations of $u$ at $t=t_0$. The figure illustrates the case of nonzero jumps. The curve $x=\gamma(t)$ acts as a barrier keeping all the curves in a compact region.}
\end{figure}
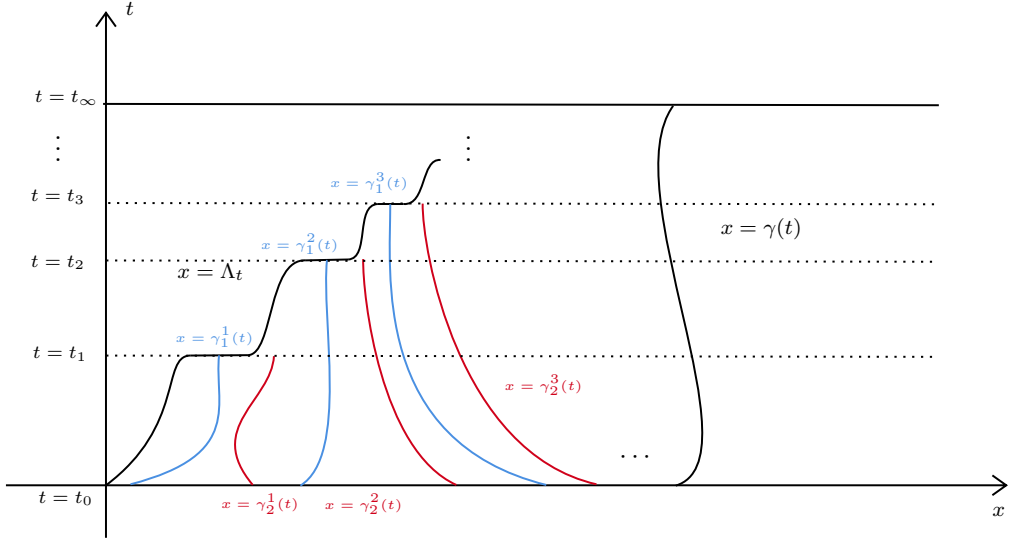

\begin{proof}[Proof of Theorems \ref{thm: jumps} and \ref{thm: stronger jumps}]
It suffices to prove Theorem \ref{thm: stronger jumps}, which is strictly stronger.
Proceeding by contradiction, assume that $t_{\infty}\in (0,\infty)$ is an accumulation point for the set of singular times, and fix $t_0<t_{\infty}$. By Lemma \ref{lem: right accum}, these times cannot accumulate to the right of $t_{\infty}$, so there exists a strictly increasing sequence $\{t_n\}_{n=1}^{\infty}$ of singular times such that $t_0<t_n\uparrow t_\infty$.
Since $\int_{\Lambda_{t_{\infty}}}^{\infty}u(x,t_{\infty})dx=\P(\tau>t_{\infty})<\infty$, there exists $x_*>\Lambda_{t_{\infty}}$ such that $u_x(x_*,t_{\infty})<0$. Therefore, by Lemma \ref{lem:global curves}, there exists a continuous curve $\gamma:[t_0,t_{\infty}] \to (0,\infty) $ such that \eqref{gam4210d} holds. Let $N\gg1$ be a fixed natural number. By Lemma \ref{lem: jump implies oscillation}, there exist continuous curves $\gamma_1^N,\gamma_2^N:[t_0,t_N] \to (0,\infty)$ such that
\begin{equation}
    \gamma_1^N(t)<\gamma^N_2(t)<\gamma(t), \quad  t\in[t_0,t_N], \quad  \gamma_1^N(t_N)>\Lambda_{t_N-},
\end{equation}
\begin{equation} \label{edqpoeq131d}
  u(\gamma_1^N(t),t),\;u(\gamma_2^N(t),t)>0, \quad u_x(\gamma_1^N(t),t)<0, \quad u_x(\gamma_2^N(t),t)>0, \quad t\in[t_0,t_N).
\end{equation}
Applying Lemma \ref{lem: jump implies oscillation} once more with $\gamma:=\gamma_1^N$, $t_*:=t_{N-1}$, and $t_{\infty}:=t_N$, we obtain continuous curves $\gamma_{1}^{N-1}, \gamma_{2}^{N-1}:[t_0,t_{N-1}]\to (0,\infty)$ such that
\begin{equation}
    \gamma_{1}^{N-1}(t)<\gamma_2^{N-1}(t)<\gamma_1^N(t)<\gamma^N_2(t)<\gamma(t), \quad  t\in[t_0,t_{{N-1}}], \quad\gamma_1^{N-1}(t_{N-1})>\Lambda_{t_{N-1}-},
\end{equation}
and \eqref{edqpoeq131d} holds with $N$ replaced by $N-1$.
Proceeding in this way, we infer that there exists a finite sequence of ordered continuous curves $\Lambda<\gamma_1^1<\gamma_2^1<\gamma_1^2<\gamma_2^2<\cdots <\gamma_1^N<\gamma_2^N<\gamma$ defined on $[t_0,t_1]$ such that 
\[u_x(\gamma_{1}^i(t),t)<0, \quad u_x(\gamma_2^i(t),t)>0,\quad t\in [t_0,t_1], \quad i\in \{1,\ldots,N\}. \]
In particular, setting $t=t_0$, we see that $u(\cdot,t_0)$ changes monotonicity at least $2N$ times on $[\Lambda_{t_0},\gamma(t_0)]$. Since $N$ can be arbitrarily large, this is a contradiction to Proposition \ref{prop: monotonicity changes}.
\end{proof}

\section{Regularity of physical solutions and their free boundary} \label{sec:regularity}
In this section, we will prove the regularity results of Theorem \ref{thm: C1 regu} and Theorem \ref{thm: classification}. We recall that regular and singular points were defined in Definition \ref{def: reg sing}, and that
\begin{equation}
    J=\{x\in \R: \Lambda_{t-}<x<\Lambda_{t} \text{ for some } t\in (0,\infty)\}.
\end{equation}
We begin by performing a classical blow-up analysis at the free boundary points located away from the jumps. We note that the following proposition would not be very useful if we had not already proved, in the previous section, that jumps do not have accumulation points.
\begin{prop}[Classification of blow-ups] \label{prop: blow-ups} Let $\Lambda$ be a physical solution to \eqref{supercooled eqs}, let $u$ be given by \eqref{u defi}, let $s$ be as in \eqref{s defi}, and let $w$ be as in \eqref{w defi}. Let $x_0\in (\Lambda_0,\alpha)$, assume that $x_0\notin \overline{J}$, and let $t_0=s(x_0)$. Then $u(x_0,t_0-)\in \{0,\alpha^{-1}\}$. Furthermore,  for $r\ll1$, letting
\[ w_r(x,t)=r^{-2}w(x_0+rx,t_0+r^2t),\]
there exists $p_2:\R \times (-\infty,0] \to \R$ such that
\begin{equation} \label{eq: blow up lim}\lim_{r\to 0^+}w_r(x,t)=p_2(x,t) \quad \text{locally uniformly in } \R \times (-\infty,0],\end{equation}
where
\begin{equation} \label{blow up types}
    p_2(x,t)=\begin{cases} \alpha^{-1}(x^{+})^2 & \text{if } u(x_0,t_0-)=0, \\
    -\alpha^{-1}t & \text{if } u(x_0,t_0-)=\alpha^{-1}.        
    \end{cases}
\end{equation}
The points $x_0 \in (\Lambda_0,\alpha)\backslash \overline{J}$ such that $u(x_0,s(x_0)-)=0$ form an open subset $\tilde{U}$ of $(\Lambda_0,\alpha)\backslash \overline{J}$ with countable complement, and we have
\[ s\in C^{\infty}(\tilde{U}),\quad \Lambda\in C^{\infty}(s(\tilde{U})),\]
with $s'>0$ in $\tilde{U}$.
\end{prop}
\begin{proof} Since $x_0\notin \overline{J}$, \eqref{nu defi} and Proposition \ref{prop: w eq} imply that, in some neighborhood of $(x_0,t_0)$, $w$ solves
\begin{equation} \begin{cases}w_t-\frac{1}{2}w_{xx}=-\alpha^{-1}\chi_{w>0} &  \\
    w_t \leq 0, \;\{w>0\}=\{w_t<0\}.       
    \end{cases}       \end{equation}
Therefore, the classical theory for the parabolic obstacle problem (see \cite[Prop. 5.5, Thm. 3.9]{BlaDolMon}) implies that \eqref{eq: blow up lim} holds for some uniquely determined $p_2$, where 
\begin{equation} \label{trichotomyq2-0ewod}
    p_2=\alpha^{-1}(x^+)^2 \quad \text{or} \quad p_2=\alpha^{-1}(x^-)^2 \quad \text{ or} \quad p_2=\alpha^{-1}(-mt+(1-m)x^2), \quad m\in [0,1]
\end{equation}
 So to prove \eqref{blow up types}, we only need to show that, if $p_2$ is of the third type in \eqref{trichotomyq2-0ewod}, then necessarily $\alpha^{-1}m=u(x_0,t_0-)$ and $m=1$, and that the second possibility in \eqref{trichotomyq2-0ewod} cannot occur. Note first that \eqref{eq: blow up lim} implies the expansion
\begin{equation} \label{taylor dpoqkwdp} w(x,t)=\alpha^{-1}(m(t_0-t)+(1-m)(x-x_0)^2)+o((t_0-t)+|x-x_0|^2),  \quad t\leq t_0.\end{equation}
Therefore, in particular, letting $x=x_0$ we have the one-sided expansion
\[w(x_0,t)=\alpha^{-1}m(t_0-t)+o((t_0-t)),  \quad t\leq t_0,\]
as $t \uparrow t_0$. Recalling that $w_t=-u$, L'H\^opital's rule and Lemma \ref{lem: lim past regu} therefore yield
\[u(x_0,t_0-)=\alpha^{-1}m.\]
On the other hand, if we assume that $m<1$, then \eqref{taylor dpoqkwdp} implies
\[w(x,t_0)=\alpha^{-1}(1-m)(x-x_0)^2+o(|x-x_0|^2).\]
But this would mean that $u(x,s(x_0))>0$ for $0<|x-x_0|\ll1$, which is clearly not possible by the geometry of the free boundary (in fact, $u(x,s(x_0))=0$ for $x<x_0$). Thus $m=1$, and $u(x_0,t_0-)=\alpha^{-1}$ whenever $p_2$ has the third form in \eqref{trichotomyq2-0ewod}. By a similar argument, we rule out the possibility  that $p_2=\alpha^{-1}(x^{-})^2$, since it would imply that $u(x,s(x_0))>0$ for $0<x_0-x\ll1$. 

We have shown that $\tilde{U}$ equals the set of points $x_0 \in (\Lambda_0,\alpha)\backslash \overline{J}$ such that the blow-up $p_2$ has the first form in \eqref{blow up types}. Thus, the fact that $\tilde{U}$ is open and that $\Lambda \in C^{\infty}(s(\tilde{U}))$ is a classical result about the parabolic obstacle problem (see, for instance, \cite[Lem.\ 4.2]{BlaDolMon}, \cite[Thm. II]{CafPetSha}). On the other hand, by Lemma \ref{lem: right accum} and Proposition \ref{prop: u nondegeneracy}, for any $x_0 \in \tilde{U}$, and some small constant $c_0>0$,
\begin{equation}
    \dot \Lambda_{s(x_0)}=\lim_{t\downarrow s(x_0)}\dot \Lambda_{t}=\frac{\alpha}{2}\lim_{t\downarrow s(x_0)}u_x(\Lambda_t,t) \geq \frac{\alpha}{2}c_0>0,
\end{equation}
which implies that $s=\Lambda ^{-1}$ is also smooth in $\tilde{U}$.
\end{proof}
We now show that the regular and singular points, as given by Definition \ref{def: reg sing}, exhaust the free boundary. We also show that, thanks to Theorem \ref{thm: jumps}, the set $\partial J$ (in the subspace topology on $(\Lambda_0,\alpha)$) consists exclusively of the endpoints of individual jumps.
\begin{lem} \label{lem: J topology} Let $\Lambda$ be a physical solution to \eqref{supercooled eqs} and let $s$ be given by \eqref{s defi}. Let $R$ and $S$ be, respectively, the regular and singular points. Then $R\cap S= \emptyset$ and $R\cup S=(\Lambda_0,\alpha)$. Moreover, if $\overline{J}$ denotes the closure of $J$ relative to $(\Lambda_0,\alpha)$, then
\begin{equation} \label{eq: J closure}\overline{J}=\{x\in (\Lambda_0,\alpha): \Lambda_{s(x)-}<\Lambda_{s(x)}, \quad x\in [\Lambda_{s(x)-},\Lambda_{s(x)}]\}.\end{equation}
    
\end{lem}
\begin{proof}
The fact that $R\cap S=\emptyset$ follows directly from Lemmas \ref{lem: smooth ext}, \ref{lem: lim past regu} and \ref{lem: u large jump}. Assume that $x_0 \in \partial J$. By Theorem \ref{thm: jumps}, jump times form a discrete set, so we must have $\Lambda_{s(x_0)-}<\Lambda_{s(x_0)}$ and $x_0\in \{\Lambda_{s(x_0)-},\Lambda_{s(x_0)}\}$. This proves that \eqref{eq: J closure} holds, and, by definition of singular points, $\partial J \subset S$. On the other hand, Proposition \ref{prop: blow-ups} shows that $(\Lambda_0,\alpha)\backslash{\partial J} \subset R\cup S$.
\end{proof}
We can now prove the first regularity result.
\begin{proof}[Proof of Theorem \ref{thm: classification}] By definition, the set $J$ is open, and $s'(x)=0$ in $J$, so we immediately obtain
\begin{equation}  \label{s jump smooth}   s \in C^{\infty}(J),\end{equation}
and, by Lemma \ref{lem: smooth ext}, 
\begin{equation} \label{usmthwi0dJ}
    u\in C^{\infty}(\{(x,t): x\in J, \;t \leq s(x)\}).
\end{equation}
By Lemma \ref{lem: u large jump} and \eqref{eq: J closure}, we have $U \cap \overline{J}= \emptyset$. This proves that $U=\tilde{U}$, where $\tilde{U}$ is as in Proposition \ref{prop: blow-ups}. Thus, $U$ is open, $s\in C^{\infty}(U)$, $\Lambda \in C^{\infty}(s(U))$ and $s'>0$ on $U$.  Since $R=U \cup J$, we therefore obtain from \eqref{s jump smooth} that $s\in C^{\infty}(R)$. We now claim that 
\begin{equation} \label{usmthwi0dk}
    u\in C^{\infty}(\{(x,t): x\in U, \;t \leq s(x)\}).
\end{equation}
By the smoothness of $\Lambda_t$, this is equivalent to showing that,  for $x_0 \in U$ and $\vep>0$ sufficiently small, $z(x,t)=u(x+\Lambda_t,t)$ satisfies
\begin{equation} \label{zsmoth3rk} z\in C^{\infty}(\overline{V_{\vep}}), \quad V_{\vep}=\{(x,t): |t-s(x_0)|<\vep, \;0 \leq x < \vep\} .\end{equation}
But, by \eqref{u equation} and the definition of $U$, $z\in C(V_{2\vep})$ solves, in the classical sense, 
\begin{equation}\begin{cases} z_{t}-\frac{1}{2}z_{xx}-\dot \Lambda_tz_x=0, &(x,t)\in V_{2\vep},\\z(0,t)=0 & t\in (s(x_0)-2\vep,s(x_0)+2\vep),\end{cases}\end{equation}
so \eqref{zsmoth3rk} follows by the boundary Schauder estimates, using again the fact that $\Lambda$ is smooth near $s(x_0)$ (see \cite[Thm 4.22]{LiebermanBook}).  Thus, in view of \eqref{usmthwi0dJ} and \eqref{usmthwi0dk}, we have completed the proof of \eqref{u smooth}. The first equality in \eqref{classical speed formula} is simply the chain rule applied to the identity $\Lambda_{s(x)}=x$, and the second equality follows, for instance, by differentiating  \eqref{lambda u eq} and using the smoothness of $u$ up to the boundary given by \eqref{usmthwi0dk}. 

Finally, the fact that $s(S)$ is locally finite in $(0,\infty)$ is simply Theorem \ref{thm: stronger jumps}. To see that $S$ itself is locally finite, it suffices to note that each singular time has at most two preimages under $s$; two in the jump case, and one in the non-jump case. 
\end{proof}
Our final objective will be to establish global $C^1$ regularity. We begin by showing that $s$ is differentiable at any singular point that is not an endpoint of a jump.
\begin{lem} \label{lem: singu point diff} Let $\Lambda$ be a physical solution to \eqref{supercooled eqs}, and let $s$ be given by \eqref{s defi}. Assume that $x_0\in (\Lambda_0,\alpha)$ is a singular point, and assume that $x_0 \notin \overline{J}$. Then $s$ is twice differentiable at $x_0$, and $s'(x_0)=s''(x_0)=0$.   
\end{lem}
\begin{proof} Let $t_0=s(x_0)$, and let $w_r$ be as in Proposition \ref{prop: blow-ups} so that, as $r\downarrow 0$,
\begin{equation}\label{eq:singular-blowup-limit}
w_r(x,t) \to p_2(t)=-\alpha^{-1}t \quad \text{locally uniformly in} \; \R \times (-\infty,0].
\end{equation}
We then have the expansion
\[w(x,t)=-\alpha^{-1}(t-t_0) +o(|t-t_0|+|x-x_0|^2), \quad t\leq t_0,\]
Since $s$ is continuous by Corollary \ref{cor:s-lipschitz} and $s(x)\leq t_0$ for $x\leq x_0$, setting $t=s(x)$ yields
\begin{equation} \label{s exp one s}
  s(x)= s(x_0) + o(|x-x_0|^2), \quad x\leq x_0
\end{equation}
In view of Proposition \ref{prop: w eq}, the rescalings have locally uniform bounds on $(w_r)_{xx}$ and $(w_r)_t$. Since $w_r\geq0$ and $w_r(0,0)=0$, we also have $(w_r)_x(0,0)=0$. Hence the bound on $(w_r)_{xx}$ gives $|w_r(x,0)|\leq Cx^2$ locally uniformly, and then the bound on $(w_r)_t$ gives $|w_r(x,t)|\leq C(|x|^2+|t|)$ locally uniformly. By standard interpolation applied to each time slice, $(w_r)_x$ is locally uniformly bounded. Thus, up to a subsequence, there exists $w_0 \in C(\R\times \R)$ such that
\[w_r(x,t) \to w_0(x,t) \quad \text{ locally uniformly in}\quad \R \times \R.\]
The same bounds pass to the limit, and $w_0\geq 0$, $(w_0)_t\leq 0$, and $w_0(x,0)=p_2(0)\equiv0$. By \eqref{eq:singular-blowup-limit}, $w_0(x,t)=p_2(t)$ for $t\leq0$. For $t>0$, the preceding positivity and monotonicity give $w_0(x,t)=0$. Hence $w_0(x,t)=p_2(t)^+$. Thus the limit is independent of the subsequence, and the convergence holds for the entire family $\{w_r\}$.

Note that, for sufficiently small $r$, $w_r$ solves the  obstacle problem
\[(w_r)_t-\frac{1}{2}(w_r)_{xx}=-\alpha^{-1}\chi_{w_r>0}, \quad (x,t)\in Q_2.\]
We use the following standard nondegeneracy property: if $(x_1,t_1)\in\overline{\{w_r>0\}}$ and $Q_\rho^-(x_1,t_1)\subset Q_2$, then
\begin{equation}\label{eq:wr-nondeg}
\sup_{Q_\rho^-(x_1,t_1)}w_r\geq\frac{1}{2\alpha}\rho^2.
\end{equation}
For completeness, we recall the short proof, following \cite[Lem.\ 5.1]{CafPetSha}. First suppose that $(x_1,t_1)\in\{w_r>0\}$, and define
\begin{equation}\label{eq:wr-barrier-nondeg}
\phi(x,t):=w_r(x,t)-w_r(x_1,t_1)-\frac{1}{2\alpha}\left((x-x_1)^2+t_1-t\right)
\end{equation}
Then $\phi$ is caloric in $\{w_r>0\}\cap Q_\rho^-(x_1,t_1)$. Since $\phi(x_1,t_1)=0$ and $\phi<0$ on $\partial\{w_r>0\}\cap Q_\rho^-(x_1,t_1)$, the maximum principle gives a point on $\partial_pQ_\rho^-(x_1,t_1)$ where $\phi\geq0$. On this parabolic boundary, $(x-x_1)^2+t_1-t\geq\rho^2$, so \eqref{eq:wr-barrier-nondeg} gives \eqref{eq:wr-nondeg}. The case $(x_1,t_1)\in\overline{\{w_r>0\}}$ follows by approximation.

Thus, since $w_r \to p_2^+=0$ uniformly in $(-2,2)\times[0,2]$, \eqref{eq:wr-nondeg} implies that, for any $\delta>0$, and $r$ sufficiently small, depending on $\delta$,
\[w_r \equiv 0  \quad \text{ on } [-1,1] \times [\delta,1].\]
Translating the statement back to $w$, and recalling that $w_t \leq 0$, this means that
\[w \equiv 0 \quad \text{ on } \quad  [x_0-r,x_0+r] \times [t_0+\delta r^2,\infty)\]
for $r$ sufficiently small. Thus, letting $x>x_0$ and $r=x-x_0$, we conclude that 
\[0<s(x)-s(x_0)<\delta |x-x_0|^2\]
for $|x-x_0|$ sufficiently small. Since $\delta$ can be taken arbitrarily small, this says precisely that \eqref{s exp one s} holds for $x>x_0$.
\end{proof}

Next, we show that, along regular points, the speed of the frontier $\Lambda$ tends to $+\infty$ near a singular point.
\begin{lem} \label{lem: regular point sequence} Let $\Lambda$ be a physical solution to \eqref{supercooled eqs}, let $x_0\in (\Lambda_0,\alpha)$ be a singular point, and assume that $\{x_n\}_{n=1}^{\infty}$ is a sequence of regular points such that $x_n \to x_0$ as $n \to \infty$. Then $s'(x_n)\to 0$ as $n\to \infty$.   
\end{lem}
\begin{proof} Let $t_0=s(x_0)$. Since $s'(x)=0$ for all $x\in J$, we may assume that $u(x_n,s(x_n)-)=u(x_n,s(x_n))=0$ for all $n$. Assume first that $x_n <x_0$ for all $n$, so that $s(x_n)<s(x_0)$. By Definition \ref{def: reg sing} and Lemma \ref{lem: u large jump}, $m:=u(x_0,t_0-)>0$ . Therefore, by Proposition \ref{prop: semiconvexity}, for $n$ sufficiently large, we have
\[u(x_0,s(x_n))\geq u(x_0,s(x_0)-)-C(s(x_0)-s(x_n))\geq \frac{m}{2}.\]
On the other hand, by Lemma \ref{lem: uxx ux bd} and Theorem \ref{thm: classification},
\begin{equation} \label{utaylor2e12e}u(x_0,s(x_n))\leq u(x_n,s(x_n))+u_x(x_n,s(x_n))(x_0-x_n)+\frac{C}{2}(x_0-x_n)^2=\frac{2}{\alpha s'(x_n)}(x_0-x_n)+\frac{C}{2}(x_0-x_n)^2 \end{equation}
 Thus, we obtain
 \[\frac{2}{\alpha s'(x_n)}\geq \frac{m}{2(x_0-x_n)}-\frac{C}{2}(x_0-x_n) \to +\infty \quad \text{ as } n\to \infty,\]
 and, in particular, for all sufficiently large $n$,
\begin{equation}\label{eq:regular-speed-left-bound}
s'(x_n)\leq C(x_0-x_n),
\end{equation}
which shows that $s'(x_n)\to 0$. 

Next, assume that $x_n >x_0$ for all $n$. Since $u(x_n,s(x_n)-)=0$, none of the points $x_n$ lies in $J$. Thus $x_0$ cannot be the left endpoint of a jump interval, since then all sufficiently close points to its right would belong to $J$. Therefore $\Lambda_{s(x_0)}=x_0$ and $u(x_0+\delta,s(x_0))>0$ for small $\delta>0$. By Lemmas  \ref{lem: uxx ux bd} and \ref{lem: smooth ext}, if $\delta>0$ is sufficiently small,
\begin{equation} \label{ssqeqprff1} u(x_0+\delta,s(x_0))=u(x_0+\delta,s(x_0)-)\geq u(x_0,s(x_0)-)-C\delta \geq \frac{m}{2}.\end{equation}
On the other hand, $x_n \in (x_0,x_0+\delta)$ for sufficiently large $n$, and as in \eqref{utaylor2e12e} we have
\begin{equation}  \label{ssqeqprff2} u(x_0+\delta,s(x_n))\leq \frac{2}{\alpha s'(x_n)}\delta + \frac{C}{2} \delta^2.\end{equation}
Thus, using \eqref{ssqeqprff1} and \eqref{ssqeqprff2} and letting $n\to \infty$, we obtain
\[\frac{m}{2}\leq \frac{2}{\alpha \limsup_{n\to \infty} s'(x_n)}\delta+\frac{C}{2}\delta^2. \]
Since $\delta$ can be taken arbitrarily small, we infer that $\limsup_{n\to \infty} s'(x_n)=0$. 
\end{proof}
\begin{rem}\label{rem:vanishing-rates}
The preceding estimates also have a simple interpretation in the original time parametrization. Suppose that $t_0>0$ is a singular time at which $\Lambda$ is continuous. Lemma \ref{lem: singu point diff} gives
\[
s(x)-t_0=o(|x-\Lambda_{t_0}|^2).
\]
Substituting $x=\Lambda_t$ and using $s(\Lambda_t)=t$, we obtain that the frontier moves faster than the parabolic scale at time $t_0$:
\begin{equation}\label{eq:continuous-singular-faster-than-sqrt}
\frac{|\Lambda_t-\Lambda_{t_0}|}{\sqrt{|t-t_0|}}\to \infty \quad \text{as } t\to t_0.
\end{equation}
At the endpoints of a jump one obtains corresponding estimates from both past and future. That is, if $t_0>0$ is a jump time, then there are constants $c,\vep>0$ such that
\begin{equation}\label{eq:jump-endpoint-sqrt-separation}
\Lambda_{t_0-}-\Lambda_t\geq c\sqrt{t_0-t}\quad \text{for }t\in(t_0-\vep,t_0),
\qquad
\Lambda_t-\Lambda_{t_0}\geq c\sqrt{t-t_0}\quad \text{for }t\in(t_0,t_0+\vep).
\end{equation}
Equivalently, in terms of the freezing time, $t_0-s(\Lambda_{t_0-}-h)\leq Ch^2$ and $s(\Lambda_{t_0}+h)-t_0\leq Ch^2$ for small $h>0$. After shrinking the neighborhood, the local finiteness of $S$ from Theorem \ref{thm: classification} ensures that there are no singular points immediately to the left of $\Lambda_{t_0-}$. The first estimate then follows by the argument leading to \eqref{eq:regular-speed-left-bound}, with $x_n$ replaced by a nearby point $y<\Lambda_{t_0-}$, and then integrating.

For the second estimate, Lemma \ref{lem: u large jump} gives $u(\Lambda_{t_0},t_0-)>0$. Thus $u(x,t_0-)\geq \eta$ on some interval $(\Lambda_{t_0},\Lambda_{t_0}+\rho)$, with $\eta,\rho>0$. Choose $K>\eta^{-1}\alpha^{-1}$ and then $A>0$ so large that $\eta\int_0^K \P(\inf_{0\leq q\leq A}(B_q+r)\leq0)\,dr>\alpha^{-1}$, which is possible by recurrence of one-dimensional Brownian motion. Then, for $0<h<\rho/K$, the Markov property, the monotonicity of $\Lambda$, and Brownian scaling give
\begin{multline}\label{eq:right-jump-endpoint-hitting}
\alpha^{-1}(\Lambda_{t_0+Ah^2}-\Lambda_{t_0})
\;=\;\P(t_0<\tau\leq t_0+Ah^2)\\
\geq \int_0^{Kh}u(\Lambda_{t_0}+z,t_0-)
\P\left(\inf_{0\leq q\leq Ah^2}(B_q+z)\leq0\right)dz\\
\geq \eta h\int_0^K
\P\left(\inf_{0\leq q\leq A}(B_q+r)\leq0\right)dr.
\end{multline}
By the choice of $K$ and $A$, the last expression in \eqref{eq:right-jump-endpoint-hitting} is larger than $\alpha^{-1}h$. Hence $\Lambda_{t_0+Ah^2}>\Lambda_{t_0}+h$, which gives the second inequality in \eqref{eq:jump-endpoint-sqrt-separation}.

Combining the second inequality in \eqref{eq:jump-endpoint-sqrt-separation} with \cite[Props.~2.1--2.2]{delarue}, we see that the estimate gives the sharp post-jump behavior in the strictly subcritical case $u(\Lambda_{t_0},t_0-)<\alpha^{-1}$. It is not always sharp in the critical case $u(\Lambda_{t_0},t_0-)=\alpha^{-1}$, as illustrated by \cite[Rem.~1.3]{delarue}. For the pre-jump behavior, corresponding to the first inequality in \eqref{eq:jump-endpoint-sqrt-separation}, the analogous sharpness question seems naturally suited to a finer obstacle-problem blow-up analysis, but we do not pursue this further.
\end{rem}
In view of the preceding lemmas, we may now prove Theorem \ref{thm: C1 regu}.
\begin{proof}[Proof of Theorem \ref{thm: C1 regu}] The $C^{\infty}$ claim in the set $R$ follows from Theorem \ref{thm: classification}, so we only need to prove that $s$ is globally of class $C^1$. We first claim that $s$ is everywhere differentiable in $(\Lambda_0,\alpha)$. By Theorem \ref{thm: classification} and Lemma \ref{lem: singu point diff}, it suffices to show that $s$ is differentiable at every singular point $x_0\in \overline{J}$. Given that $J$ consists of regular points, it follows that $x_0\in \partial J$, which, by Lemma \ref{lem: J topology}, means that $\Lambda_{s(x_0)-}<\Lambda_{s(x_0)}$ and $x_0\in \{\Lambda_{s(x_0)-},\Lambda_{s(x_0)}\}$. Furthermore, by Theorem \ref{thm: jumps}, there must exist a small $\delta>0$ such that $(x_0-\delta,x_0+\delta) \cap \partial{J} =\{x_0\}$. But then, in view of Proposition \ref{prop: blow-ups} and Lemmas \ref{lem: singu point diff} and \ref{lem: regular point sequence}, $s$ is differentiable in $(x_0-\delta,x_0+\delta)\backslash\{x_0\}$, and $\lim_{x \to x_0} s'(x)=0$. By L'H\^opital's rule, we infer that $s$ is differentiable at $x_0$, and $s'(x_0)=0$. This shows that $s$ is everywhere differentiable and, in fact, $s'(x_0)=0$ for every singular point $x_0$. By Theorem \ref{thm: classification}, $s'$ is continuous at any regular point. Finally, by another application of Lemma \ref{lem: regular point sequence}, we conclude that $s'$ is continuous at every singular point, and therefore in all of $(\Lambda_0,\alpha)$.
\end{proof}
\subsection*{Acknowledgments}The author would like to thank I. C. Kim for valuable discussions, comments, and suggestions. The author also thanks the anonymous referees for their careful reading of the paper and their very useful suggestions.

\end{document}